\documentclass[11pt]{article}
\usepackage{amsmath, amssymb, amsthm, amsfonts}

\numberwithin{equation}{section}
\usepackage{graphicx}

\graphicspath{{./figs/}}
\usepackage{appendix}
\usepackage{hyperref}
\usepackage[a4paper, margin=1in]{geometry}
\usepackage{authblk}
\usepackage{booktabs}

\newtheorem{thm}{Problem}

\begin{document}

\title{AI for Mathematics: Progress, Challenges, and Prospects}


\author[1]{Haocheng Ju}
\author[2,3,4,5]{Bin Dong\thanks{Corresponding author}}

\affil[1]{School of Mathematical Sciences, Peking University, Beijing 100871, China}
\affil[2]{Beijing International Center for Mathematical Research, Peking University, Beijing 100871, China}
\affil[3]{Center for Machine Learning Research, Peking University, Beijing 100871, China}
\affil[4]{Center for Intelligent Computing, Great Bay Institute for Advanced Study, Great Bay University, Dongguan 523000, China}
\affil[5]{Zhongguancun Academy, Beijing 100094, China}

\maketitle


 {\begin{center}
\parbox{14.5cm}{\begin{abstract}
AI for Mathematics (AI4Math) has emerged as a distinct field that leverages machine learning to navigate mathematical landscapes historically intractable for early symbolic systems. While mid-20th-century symbolic approaches successfully automated formal logic, they faced severe scalability limitations due to the combinatorial explosion of the search space. The recent integration of data-driven approaches has revitalized this pursuit. In this review, we provide a systematic overview of AI4Math, highlighting its primary focus on developing AI models to support mathematical research. Crucially, we emphasize that this is not merely the application of AI to mathematical activities; it also encompasses the development of stronger AI systems where the rigorous nature of mathematics serves as a premier testbed for advancing general reasoning capabilities. We categorize existing research into two complementary directions: problem-specific modeling, involving the design of specialized architectures for distinct mathematical tasks, and general-purpose modeling, focusing on foundation models capable of broader reasoning, retrieval, and exploratory workflows. We conclude by discussing key challenges and prospects, advocating for AI systems that go beyond facilitating formal correctness to enabling the discovery of meaningful results and unified theories, recognizing that the true value of a proof lies in the insights and tools it offers to the broader mathematical landscape.

\end{abstract}}\end{center}}


\section{Introduction}

The automation of mathematical reasoning has been a central objective of artificial intelligence (AI) since its inception. The dream of mechanizing reasoning predates the digital computer, tracing back to the 1920s when David Hilbert proposed a program to formalize all of mathematics, aimed at proving every theorem within a consistent axiomatic system. While this optimism was theoretically challenged by Kurt Gödel’s incompleteness theorems in 1931 \cite{godel1931formally}, which demonstrated that any sufficiently rich formal system is necessarily incomplete, Gödel’s results did not close the subject. As noted by Ulam \cite{bams/1183522369}, von Neumann viewed these discoveries as a prompt to rethink the role of formalism rather than abandon it. Later research has further shown that large portions of classical mathematics admit finitistic reductions \cite{simpson1988partial}, keeping the dream of partial mechanization alive.

With the advent of digital computers, these theoretical ideas were translated into practice. In the 1950s, Martin Davis implemented a decision procedure for Presburger arithmetic \cite{Clarke_1966,Davis1983ThePA}, arguably the first instance of a computer mechanically verifying logical statements. Shortly after, Newell, Simon, and Shaw developed the ``Logic Theorist'' \cite{1056797}, widely considered the first true AI program, which performed theorem proving in symbolic logic. A particularly influential milestone in this symbolic era was Wen-Tsun Wu’s method of mechanical theorem proving in geometry \cite{wen1986basic,wu2012mechanical}. By translating geometric problems into systems of algebraic equations and applying the characteristic set method, Wu demonstrated that complex logical consequences could be derived algorithmically.

However, most of these symbolic methods faced a critical bottleneck: the combinatorial explosion of the search space. As proofs grew in complexity, the number of possible logical paths expanded exponentially, rendering exhaustive search intractable. In recent decades, the integration of machine learning has revitalized this pursuit, offering data-driven approaches to navigate these complex landscapes. Since the 2010s, connectionist AI \footnote{Connectionist AI is an AI paradigm inspired by the interconnected structure of neurons in the human brain. It enables perception, reasoning, and decision-making through multi-layer networks of nodes that learn patterns from data by optimizing connection weights.} has risen to prominence, achieving remarkable success in computer vision and natural language processing. Mathematicians have subsequently begun leveraging these models to identify patterns that guide human intuition \cite{davies2021advancing,hashemi2025can,dong2024machine}, to construct counterexamples via reinforcement learning (RL) \cite{wagner2021constructions,Gukov_2025,shehper2025what}, and to train neural theorem provers \cite{whalen2016holophrasm,bansal2019holist}. More recently, the rapid progress of large language models (LLMs) has enabled the generation of new mathematical constructions \cite{romera2024mathematical, novikov2025alphaevolve,georgiev2025mathematical}, autoformalization \cite{wu2022autoformalization,mathinc_gauss}, and collaborative theorem proving \cite{zhou2025solving}.

This convergence has given rise to the interdisciplinary field of AI for Mathematics (AI4Math). We emphasize that AI4Math is not merely the application of AI tools to mathematical tasks; it also encompasses the development of stronger AI systems where the rigorous nature of mathematics serves as a premier testbed for advancing general reasoning capabilities. Broadly speaking, as illustrated in Figure \ref{fig:landscape}, research in this field can be categorized into two complementary directions:

\begin{itemize}
    \item \textbf{Problem-Specific Modeling:} This involves the design of specialized architectures tailored to specific research questions or narrow classes of problems, such as guiding intuition in knot theory or reasoning within closed geometric systems. Beyond their high effectiveness for targeted tasks, these models typically require significantly less data and computational resources, making them accessible to a broader range of researchers. However, they rarely transfer to other domains without substantial modification.
    \item \textbf{General-Purpose Modeling:} This focuses on the development of foundation models, ranging from specialized math-specific language models to general-purpose reasoning engines, designed to support broader workflows across diverse mathematical areas. While offering versatility, these approaches demand massive training datasets, substantial computational resources, and significant engineering expertise. Furthermore, they may not achieve the same degree of specificity or effectiveness as specialized models when applied to narrowly defined mathematical problems. This category encompasses advances in natural language reasoning, the bridging of informal and formal mathematics through autoformalization, and the construction of agentic systems capable of automated theorem proving and information retrieval.
\end{itemize}

It is worth noting that the scope of AI4Math technically extends beyond logical reasoning to include AI for Computational Mathematics and Scientific Computing. This domain focuses on building AI models to assist in numerical computations, such as for solving partial differential equations (PDEs), optimization, and inverse problems. The roots of this approach trace back to the 1980s and 1990s, where researchers explored the use of shallow neural networks to approximate solutions to differential equations. However, the field experienced a renaissance around 2015 with the advent of deep learning. While these computational advances constitute a vital pillar of the broader AI4Math landscape, this review restricts its focus to mathematical reasoning—encompassing discovery, formalization, and proof. We refer readers interested in the computational aspects of AI4Math to other comprehensive surveys \cite{arridge2019solving,karniadakis2021physics,brunton2024promising,weinan2021dawning,dong2022mathematical,chen2022learning}.

In this paper, we provide a systematic overview of the progress, challenges, and prospects of AI4Math. Our goal is not to be exhaustive, but to highlight representative works that illustrate the field's evolution. Readers interested in specific subareas may consult existing surveys on automated reasoning in geometry \cite{chou2001automated,wu2007mathematics} or deep learning for theorem proving \cite{li2024a}. Section \ref{sec:psm} examines problem-specific modeling, while Section \ref{sec:grm} reviews general-purpose modeling. Finally, Section \ref{sec:cp} discusses the key challenges ahead, advocating for systems that move beyond simple verification toward the discovery of deep mathematical insights.

\begin{figure}[tb]
\centering
\includegraphics[width=0.95\linewidth]{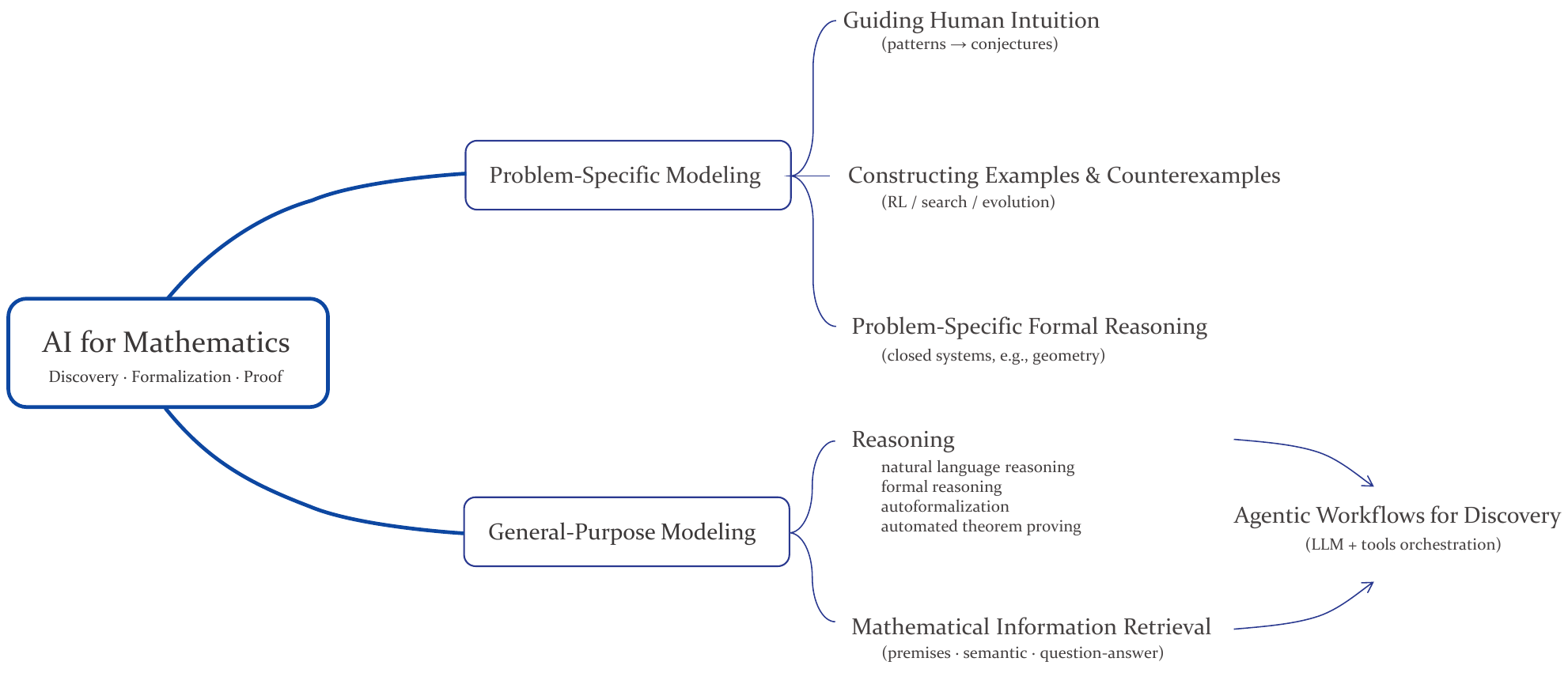}
\caption{Landscape of AI4Math research directions. AI4Math research can be broadly divided into two complementary paradigms: \textit{problem-specific modeling} and \textit{general-purpose modeling}. Problem-specific modeling encompasses three major lines of work: using machine learning to guide human intuition, constructing examples and counterexamples, and performing formal reasoning in closed systems. General-purpose modeling focuses on mathematical reasoning and mathematical information retrieval; together, these capabilities underpin emerging agentic workflows for mathematical discovery. This diagram is intended to organize and clarify conceptual relationships among research directions, rather than to provide a one-to-one correspondence with the sections of this review.}
\label{fig:landscape}
\end{figure}

\section{Problem-Specific Modeling}\label{sec:psm}
With the rapid advancement of data-driven techniques, researchers started to design specialized machine learning models tailored to specific mathematical research problems. These efforts generally fall into three distinct directions: identifying patterns in high-dimensional data to guide human intuition and inspire new conjectures; constructing examples or counterexamples to rigorously test or disprove mathematical hypotheses; and performing formal reasoning within closed axiomatic systems, such as Euclidean geometry. In this section, we review recent developments across these three areas and discuss their respective advantages and limitations.

\subsection{Guiding Human Intuition through Machine Learning}

One of the earliest examples of utilizing machine learning to assist in formulating mathematical conjectures is \cite{carifio2017machine}. In this study, the authors employed linear regression to predict the rank of the geometric gauge group across a large ensemble of F-theory geometries, successfully rediscovering an existing conjecture regarding gauge group rank. Building on this, they applied logistic regression to a classification problem involving $E_6$ gauge groups; through attribution analysis of the model, they formulated a genuinely new conjecture that was subsequently proved by human mathematicians.

However, the pivotal study that demonstrated the broad potential of deep learning in mathematical research is \cite{davies2021advancing}. The central contribution of this work is a systematic framework designed to accelerate the conjecture generation process—traditionally a time-consuming cycle where mathematicians hypothesize relationships, attempt proofs, and refine their ideas iteratively. In the workflow proposed by \cite{davies2021advancing}, mathematicians begin by hypothesizing that a relationship may exist between two mathematical objects. A specially designed neural network is then trained to predict one quantity from the features of the other. Attribution methods are subsequently used to identify the most influential input components, thereby guiding mathematicians toward more precise and refined conjectures. This cycle repeats until a mathematically meaningful statement is formulated. Using this approach, the authors discovered a new relationship between algebraic and geometric invariants in knot theory \cite{davies2024signature} and proposed a candidate algorithm predicted by the combinatorial invariance conjecture for symmetric groups \cite{blundell2022towards}.

Inspired by this paradigm, \cite{dong2024machine} designed another AI-guided intuition framework to the study of \textit{affine Deligne--Lusztig varieties (ADLV)}. This work not only independently rediscovered the classical virtual dimension formula in arithmetic geometry but also established a novel, precise lower-bound theorem. By providing a quantitative characterization that fills an important theoretical gap, this result demonstrates the efficacy of \textit{AI-guided intuition} in facilitating rigorous discoveries within deep areas of pure mathematics.

Beyond refining conjectures, machine learning has also proven effective in uncovering entirely new mathematical phenomena. A notable example is \cite{he2025murmurations}, where the authors represented elliptic curves as vectors and trained a logistic regression classifier to distinguish curves of different ranks. To interpret the classifier's strong performance, the authors performed Principal Component Analysis (PCA) on the vector representations. Upon plotting the averages of Frobenius traces for elliptic curves of fixed rank over a given conductor range, they revealed a surprising oscillatory pattern, which they termed ``murmurations." This phenomenon has since been the subject of in-depth theoretical study \cite{zubrilina2025murmurations,lee2025murmurations,bober2025murmurations}. A growing body of literature continues to leverage machine learning to discover or rediscover relationships across various domains \cite{carifio2017machine,KLAEWER2019438,PhysRevD.103.126014,doi:10.1142/S0217751X25500976,PhysRevD.105.066002,Lee01052025,doi:10.1142/S2810939223500016,PhysRevD.110.046016,shalyt2024unsupervised,lanard2025algorithm,lee2025machines}.

\subsection{Constructing Examples and Counterexamples}

One of the pioneering works to utilize machine learning for constructing examples and counterexamples is \cite{wagner2021constructions}. In this study, the author encodes graphs as 0--1 sequences and applies the deep cross-entropy method from RL to search for graph constructions that serve as counterexamples to existing conjectures. Following this precedent, RL has been adopted for a variety of structural problems. For instance, \cite{berczi2023ml} models the Hironaka game, a framework central to the resolution of singularities in algebraic geometry, as a Markov decision process (MDP). By combining Monte Carlo tree search with Deep Q-Networks, the authors successfully trained an agent to replace singular points with smooth ones, achieving near-optimal resolutions for du Val singularities. Similarly, \cite{shehper2025what} investigates the \textit{Andrews--Curtis (AC)} conjecture. After first employing classical search algorithms to identify AC triviality for infinite subfamilies of the \textit{Miller--Schupp} series and to obtain length reductions in the \textit{Akbulut--Kirby} series, the authors formulate the general problem as an MDP. They trained an RL agent on problem instances of varying difficulty, ultimately discovering two novel AC trivializations of balanced presentations that had eluded classical search methods.

Beyond RL, a diverse array of other machine learning techniques has been proposed for mathematical construction. \cite{alfarano2024global} trains a Transformer model on synthetic data to predict Lyapunov functions for stable dynamical systems; the trained model is then used to discover novel Lyapunov functions for non-polynomial systems. In combinatorics, \cite{charton2024patternboost} introduces an iterative \textit{bootstrapping} procedure. This method begins by generating candidate constructions via a local search algorithm, trains a neural network on the highest-scoring candidates, and then samples new seeds from the network to initialize the next round of local search. Using this approach, the authors successfully discovered a counterexample to a 30-year-old conjecture. \cite{bérczi2026flowbasedextremalmathematicalstructure} proposes a geometry-aware conditional flow-matching model for continuous geometric optimization problems, interleaving flow integration with projection onto the constraint manifold. The method initializes a training set using stochastic relaxation with perturbations, trains a conditional flow-matching model with geometric penalties, and employs geometry-aware sampling together with reward-guided policy optimization to iteratively refine configurations via local search. This approach discovers configurations that match or exceed the best known results in problems such as circle packing maximizing the sum of radii and the Heilbronn triangle problem. Additionally, \cite{BERGLUND2024138504} applies genetic algorithms to generate reflexive polytopes in dimensions 2 through 5, identifying several previously unknown polytopes in dimension 5. Other works leveraging machine learning for the construction of examples or counterexamples include \cite{chervov2025cayleypy, chervov2025machine, Gukov_2025,swirszcz2025advancing,douglas2025diffusion}.

As LLMs continue to grow more capable, LLM-based agents have also improved, and several recent works have demonstrated their potential for discovering new mathematical constructions. \textit{FunSearch} \cite{romera2024mathematical} searches for programs that generate desired constructions using an evolutionary approach. For problems equipped with a well-defined evaluator, FunSearch employs an off-the-shelf LLM to iteratively evolve low-scoring candidate programs into higher-scoring ones. This is achieved by maintaining a large, diverse pool of programs and repeatedly prompting the LLM to improve upon earlier candidates. Using this method, FunSearch discovered new constructions of large cap sets that surpass the best previously known results in extremal combinatorics. Building on this line of work, \textit{AlphaEvolve} \cite{novikov2025alphaevolve} uses a stronger LLM and extends the evolutionary process to entire code files rather than single functions, while also optimizing multiple metrics simultaneously. AlphaEvolve has produced improved constructions for several problems, including the \textit{Minimum Overlap Problem} and the \textit{Kissing Numbers problem} in 11 dimensions. Open-source implementations inspired by AlphaEvolve include \textit{OpenEvolve} \cite{openevolve}, \textit{ShinkaEvolve} \cite{lange2025shinkaevolve}, and \textit{DeepEvolve} \cite{liu2025scientific}. AlphaEvolve-style agents are particularly suitable for finding new constructions in mathematical problems that can be approached by writing code and evaluated with a well-defined scoring function.

\subsection{Problem-Specific Formal Reasoning}

\textit{AlphaGeometry} \cite{trinh2024solving} represents a neuro-symbolic approach to solving Olympiad-level Euclidean geometry problems. It integrates a symbolic geometric reasoning engine with a language model designed to suggest auxiliary constructions. The symbolic component, based on a deductive database (DD) \cite{chou2000deductive} and algebraic rules (AR), exhaustively derives the deduction closure of a given set of premises. Since purely symbolic deduction cannot introduce new geometric objects, a capability often required for complex proofs, the system utilizes a language model to propose useful auxiliary points. To circumvent the scarcity of human proof data, the authors generated a large-scale dataset of synthetic proof graphs by alternating symbolic deduction with random point insertions, subsequently extracting minimal proofs via traceback. During inference, the system operates in a loop: the symbolic engine expands the deduction closure, the model proposes high-probability auxiliary points, and the cycle repeats until the target conclusion is reached. This architecture allows AlphaGeometry to significantly outperform heuristic-based systems, reaching performance comparable to an IMO silver medalist.

Its successor, \textit{AlphaGeometry2} \cite{chervonyi2025gold}, further advances this paradigm through enhancements in both expressivity and efficiency. The formal geometry language was expanded to support locus descriptions, linear geometric relations, and non-constructive statements, while the underlying symbolic engine was re-engineered for greater speed and robustness. These improvements enabled the generation of a larger, more diverse synthetic training set for the language model. Furthermore, AlphaGeometry2 introduces a novel search algorithm, the \textit{Shared Knowledge Ensemble of Search Trees (SKEST)}, which executes multiple search trees in parallel. By allowing these trees to exchange discovered information, SKEST significantly improves the exploration of the auxiliary construction space. Consequently, AlphaGeometry2 achieves gold-medalist performance on IMO-level geometry problems. Beyond neural approaches for auxiliary point generation, recent work such as \textit{HAGeo} \cite{duan2025gold} proposes a purely heuristic-based strategy that introduces auxiliary constructions with favorable geometric properties, including intersections of lines and circles, midpoints, and point reflections, and likewise achieves gold-medal-level performance. Additional work on Euclidean geometry problem solving can be found in \cite{zhang2023formalgeo,he2024fgeo,zou2024fgeo,zhang2024fgeo,zhang2024proposing}.

\subsection{Discussion}

The approaches discussed in this section, identifying patterns to guide intuition, constructing counterexamples, and performing formal reasoning in closed systems, each offer distinct advantages while presenting unique challenges.

The \textit{AI-guided intuition} paradigm is powerful because it enables mathematicians to uncover patterns in high-dimensional data that are otherwise difficult or time-consuming to detect manually, effectively reducing the search space in exploratory research. However, this approach is not universally applicable. It relies on careful problem selection, as the target question must admit the generation of sufficiently large and representative datasets. Furthermore, successful implementation demands a high barrier of dual expertise: beyond standard machine learning considerations such as architecture design and loss engineering, deep mathematical insight is indispensable for interpreting model outputs and transforming empirical correlations into rigorous mathematical theory. Ultimately, since the final proof and verification are typically carried out by human mathematicians, the degree of automation in this workflow remains limited.

Conversely, employing machine learning to construct examples and counterexamples can significantly accelerate the formulation and testing of conjectures by discovering objects that defy human intuition. Yet, this direction also faces technical hurdles, particularly regarding out-of-distribution (OOD) generalization. For instance, the backward generation method in \cite{alfarano2024global}, which constructs problems from sampled solutions, may yield training sets with specific distributions. This can present challenges for generalizing to typical problem instances and often requires carefully designed procedures, such as the diversity promoting mechanism for Lyapunov functions proposed by the authors, to ensure robust performance. Additionally, when RL is employed, mapping a mathematical problem onto an MDP is non-trivial. Defining appropriate state representations, action spaces, and reward functions can be intricate \cite{wagner2021constructions}, and the learning process may be further complicated by sparse rewards and long planning horizons \cite{shehper2025what}.

Systems like FunSearch and AlphaEvolve represent a powerful shift toward automated mathematical discovery through program evolution. These agents excel at construction-type problems, such as finding extremal configurations in combinatorics or improving algorithmic bounds, where the search space can be navigated by iteratively refining code and verified through a well-defined scoring function. However, this paradigm remains inherently constrained by its reliance on such evaluators. While highly effective at identifying specific mathematical objects (e.g., cap sets or high-dimensional sphere packings), these systems do not explicitly engage with the formulation of conceptual frameworks or structural principles that underlie general theories or connect disparate mathematical domains. Their success is therefore currently concentrated in verifiable construction, leaving the discovery of transferable abstractions and unifying insights as an open frontier.

Finally, problem-specific formal reasoning systems, such as AlphaGeometry, demonstrate that combining symbolic engines with neural language models can achieve expert-level performance in structured domains. However, the success of these systems typically relies on the availability of domain-specific symbolic engines (e.g., deductive databases for geometry \cite{chou2000deductive}) and the capability to generate large-scale synthetic data. Consequently, these architectures are often tightly tailored to their specific problem scope and may not transfer to other areas of mathematics without substantial modification.

\section{General-Purpose Modeling}\label{sec:grm}

General-purpose modeling marks a shift from specialized algorithms designed for isolated problems to adaptable systems capable of handling broad areas of mathematics. Unlike problem-specific modeling, which requires bespoke features and architectures for each new task, general-purpose approaches leverage foundation models trained on vast corpora to learn universal representations of mathematical knowledge. These models aim to support a wide spectrum of activities, from solving diverse problem sets to retrieving theorems and orchestrating complex workflows, without requiring substantial modification for each new domain.

We categorize recent progress in general-purpose modeling into three complementary directions: natural language reasoning models that leverage the intuitive power of language; formal reasoning models that ensure rigor through interaction with proof assistants; and mathematical information retrieval systems that ground reasoning in established knowledge. This section begins by analyzing the capabilities and inherent limitations of foundation models (especially LLMs), establishing the context for our detailed review of these three key areas.

\subsection{Foundation Models and LLMs}

Compared with traditional machine learning models that are trained for a single, narrowly defined task, LLMs serve as foundation models: a single architecture trained on a broad collection of data and tasks in a unified manner. Mathematically, this distinction represents a paradigm shift from function approximation to operator approximation, which is a process intimately related to \textit{meta-learning}. A traditional model is typically designed to approximate a specific, fixed mapping $f: \mathcal{X} \to \mathcal{Y}$ between Euclidean spaces, such as classifying objects in an image, by interpolating discrete data points sampled from that specific manifold.

In contrast, foundation models perform \textit{task interpolation} rather than simple data interpolation. They are trained over a vast distribution of tasks $\{T_i\}$, where each task effectively defines a distinct mapping $f_i$. By learning to generalize across this distribution, the foundation model does not merely approximate a single function; instead, it approximates a global operator (or functional) $\Psi$ acting on a function space. Given a task specification or context as input, $\Psi$ instantiates the specific mapping $f_i$ required for that context. In essence, while traditional machine learning approximates functions, foundation models approximate the operators that generate and manipulate those functions. 

A critical factor in their success is the capability to process diverse data types; tokenization converts heterogeneous inputs into a common sequence representation, and the next-token prediction objective provides a single learning rule that applies uniformly across all tasks the model encounters. Attention-based architectures are pivotal in this regime. Beyond effectively scaling with model size and data volume, the attention mechanism serves as the key engine for reasoning by enforcing long-context coherence during training. This allows the model to capture and maintain complex dependencies over extended sequences—a prerequisite for logical deduction. Through exposure to diverse domains and supervision signals, the model is forced to compress vast amounts of heterogeneous data into shared internal representations and to discover a common low-dimensional structure across tasks and languages. A natural hypothesis is that a key component of this low-dimensional structure corresponds to general-purpose reasoning abilities that can be expressed in different languages and domains.

Mathematics fits naturally into this framework. Mathematical work is governed by strict logical rules, and many mathematical tasks can be phrased as producing the next meaningful step in a calculation, derivation, or proof—exactly the kind of stepwise structure that the next-token prediction objective is designed to model. Consequently, when an LLM is trained as a foundation model on sufficiently rich mathematical and scientific corpora, the same mechanisms that support cross-domain generalization and long-context coherence can be harnessed to learn and use a wide range of mathematical reasoning abilities.

However, a fundamental gap remains between mastering standardized examinations and engaging in research-level mathematics. While current models excel at solving well-defined undergraduate problems, research demands open-ended exploration, absolute logical rigor, and the ability to navigate the ``long tail" of specialized domain knowledge—challenges where stochastic text generation often falls short. To elevate AI from a competent solver to a reliable research partner, general-purpose modeling must therefore advance beyond simple next-token prediction, integrating formal verification, semantic retrieval, and agentic workflows to bridge the divide between plausible text and rigorous truth.

\subsection{Natural Language Reasoning}\label{subsec:nlr}

Current approaches to natural language mathematical reasoning generally fall into two categories: \textit{math-specific LLMs} and \textit{general-purpose reasoning models}.

Math-specific LLMs are typically adapted from general foundation models through specialized pre-training and post-training pipelines. During pre-training, filtering pipelines \cite{shao2024deepseekmath,yang2024qwen2} extract high-quality mathematical content from web corpora (e.g., Common Crawl), textbooks, and research papers to maximize domain relevance. Post-training refines these models using supervised fine-tuning (SFT) and RL. The data for SFT is often structured as Chain-of-Thought (CoT) \cite{wei2022chain} pairs, consisting of problems and step-by-step solutions, or Tool-Integrated Reasoning (TIR) \cite{goutora} examples that incorporate external code execution. A prominent example is \textit{NuminaMath} \cite{li2024numinamath}, which won the first AIMO Progress Prize by fine-tuning on high-quality CoT and TIR datasets. While models in this category \cite{azerbayev2024llemma,shao2024deepseekmath,yang2024qwen2,ying2024internlm} excel at elementary and competition-level benchmarks (e.g., GSM8K \cite{cobbe2021training}, MATH \cite{hendrycksmath2021}, AIME), their capacity for advanced mathematics remains less explored.

Concurrently, general-purpose LLMs have achieved significant mathematical milestones driven by scale and novel inference strategies. Early iterations like GPT-3 struggled with basic arithmetic, whereas GPT-4 \cite{achiam2023gpt} achieved 92.0\% on GSM8K. A paradigm shift occurred with the introduction of \textit{test-time scaling}, where models dedicate increased computation to reasoning during inference. OpenAI's o1 model demonstrated strong performance on the AIME, and subsequent reasoning models \cite{guo2025deepseek,qwq32b,yang2025qwen3,comanici2025gemini} have further validated this approach. By 2025, enhanced reasoning models, such as Google's Gemini Deep Think, reached gold-medal performance at the International Mathematical Olympiad (IMO) using purely natural-language reasoning, marking a definitive maturity for the technology in the domain of high-school olympiad mathematics.

However, transitioning from Olympiad problems to higher mathematics presents a steeper challenge. Prior studies indicate that while GPT-4 can assist with undergraduate topics, it requires critical human oversight \cite{collins2024evaluating} and often fails at the graduate level \cite{NEURIPS2023_58168e8a}. Recent benchmarks quantify this gap: \cite{jiang2025fate} reported that DeepSeek-R1 achieves 71.0\% proof accuracy on graduate-level algebra (FATE-H) but drops to 33.0\% on PhD qualifying exams (FATE-X). Similarly, on the \textit{FrontierMath} Benchmark \cite{glazer2024frontiermath}, which consists of unpublished research-level problems, Gemini 3 Pro scores 18.75\% on the research-level split (Tier 4)\footnote{\href{https://epoch.ai/frontiermath}{https://epoch.ai/frontiermath}}, indicating that robust research-level reasoning remains an open problem.

To empirically assess the current capabilities of state-of-the-art models, we constructed a dataset of 100 problems drawn from the final exams of 11 undergraduate mathematics subjects at Peking University (PKU). We evaluated five models: GPT-4, o1, o3-mini, DeepSeek-R1, and Gemini 2.5 Pro. Sample problems and model responses are provided in Appendix \ref{sec:exam}. Human experts graded the outputs on a 0--5 scale (criteria in Table \ref{tab:score_crtr}); the normalized results are presented in Figure \ref{fig:nl_reason} (left and middle). While GPT-4 scores below 60, reasoning-enhanced models (OpenAI o-series, DeepSeek-R1, Gemini 2.5 Pro) show substantial improvement, with several exceeding a score of 90. 

Furthermore, we evaluated o3-mini on 58 problems from PKU PhD Qualifying Exams in Analysis, Probability, Algebra, and Geometry \& Topology. As shown in Figure \ref{fig:nl_reason} (right), o3-mini achieves an average score of 84.4. A closer examination of the subject-wise performance reveals a notable divergence: the model exhibits its strongest proficiency in Algebra while scoring lowest in Geometry \& Topology. Under the assumption that these exams pose a comparable level of challenge for human students, this suggests that current AI systems are comparatively more adept at handling abstract algebraic structures than tasks requiring geometric intuition. Although these results must be interpreted with caution due to potential data contamination and the specific nature of exam problems compared to open research, they provide strong evidence that top-tier models can now handle a significant portion of graduate-level mathematics.

Building on frontier models, LLM-based agents equipped with appropriate harness engineering can further extend performance toward research-level mathematics. One representative example is \textit{Aletheia} \cite{feng2026towards}, an agent built on an advanced version of Gemini Deep Think that coordinates separate generator, verifier, and reviser components. It has been applied to several research-level tasks, including the computation of eigenweights for the Arithmetic Hirzebruch Proportionality Principle \cite{feng2026eigenweights}, questions about the simplicity of the Hodge bundle \cite{patel2026simplicity}, and bounds for independence sets \cite{lee2026lower}. In addition, \textit{Aletheia} solved 6 of the 10 problems in the FirstProof benchmark \cite{feng2026aletheia}, introduced in \cite{abouzaid2026first}, a collection of research problems that had been solved by human mathematicians but were not publicly available at the time of evaluation. Another example is \textit{Rethlas} \cite{ju2026automated}, an agent designed to mimic the workflow of human mathematicians by combining exploratory reasoning with tools such as theorem retrieval. This framework automatically resolved an open problem in commutative algebra proposed by D. D. Anderson in 2014 \cite{cahen2014open}.

Synthesizing these findings, we observe a clear trajectory: the mathematical reasoning abilities of LLMs and LLM-based agents have progressed from mastering elementary calculations and high-school competitions to achieving competence in undergraduate curriculum, and are now beginning to penetrate the domain of graduate and research-level mathematics \cite{bryan2026motivic,jang2025point}.


\begin{figure}[tb]
\centering
\includegraphics[width=0.95\linewidth]{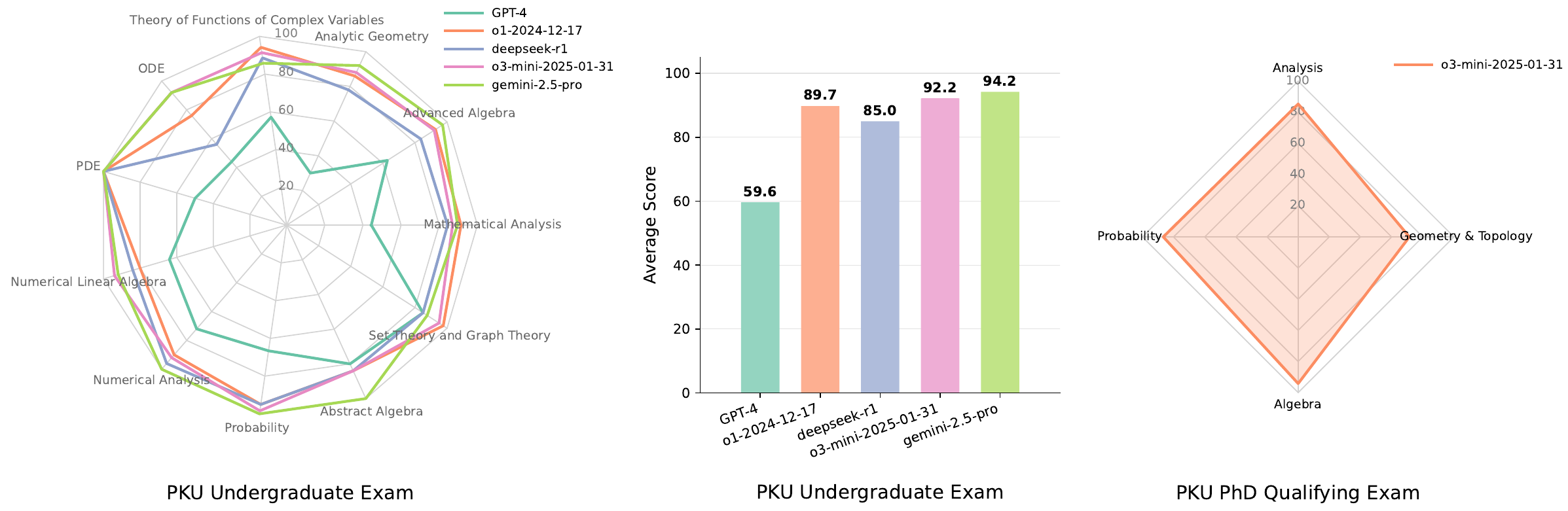}
\caption{Performance of five LLMs across different Peking University (PKU) exams. Left: Scores on 11 undergraduate subjects. Middle: Average scores of the five LLMs across these subjects. Right: o3-mini’s performance on four PhD qualifying exam fields (average score: 84.4).}
\label{fig:nl_reason}
\end{figure}

\subsection{Formal Reasoning}\label{subsec:fr}

Although top-tier LLMs and LLM-based agents can now solve certain graduate-level and even certain research-level problems, evaluating their capabilities remains a significant bottleneck involving substantial human effort. As the complexity of mathematics increases, assessment depends heavily on domain experts; yet, because natural language is inherently imprecise, even experienced mathematicians can misjudge arguments. A notable historical example occurred in 1994, when a result published in the \textit{Annals of Mathematics} claimed that the Busemann–Petty problem has a negative solution in dimensions at least four \cite{zhang1994intersection}. This conclusion was later shown to be incorrect \cite{koldobsky1998intersection,koldobsky1998advances}, and a 1999 paper established that the problem in fact possesses a positive solution in dimension four \cite{zhang1999positive}. Similarly, a paper \cite{schumacher2004quasi} published in the \textit{Annals of Mathematics} in 2004 was retracted in 2023 \cite{schumacher2023retraction} after a counterexample to its main result, discovered in 2006, was identified \cite{kollar2006non}. These cases illustrate that errors can persist even within the rigorous peer-review processes of top journals. To enable fast and reliable verification of mathematical reasoning, research must therefore move toward a more mechanical and unambiguous framework. Formal systems provide exactly such a foundation. In this section, we discuss the benefits of formal systems for mathematical research and their utility in enhancing the reasoning abilities of LLMs, followed by a review of recent progress in automated theorem proving and autoformalization.

\subsubsection{Formal Systems}

Formal systems provide a precise symbolic language paired with rigorously defined mechanisms for constructing and verifying proofs. Various systems exist, distinguished by their underlying logical foundations: HOL systems (e.g., HOL Light, Isabelle/HOL) utilize simple type theory; Coq (now known as Rocq) and Lean employ dependent type theory; Metamath operates on first-order logic with explicitly specified axioms; and Mizar is grounded in Tarski–Grothendieck set theory. Once a mathematical argument is translated into the formal language of an interactive theorem prover (ITP), it can be verified with absolute rigor. Had the erroneous 1994 result regarding the Busemann--Petty problem been formalized in such a system, the underlying logical flaws would likely have been detected immediately, preventing the publication of an incorrect claim.

Beyond their intrinsic value to mathematical correctness, formal systems offer a critical advantage for AI development: they provide reliable, verifiable supervision. Unlike elementary mathematics, where answers are often numeric and easily checked, advanced proof-oriented problems lack simple verifiers, making it difficult to generate reliable training signals. Interactive theorem provers bridge this gap by providing precise feedback on the validity of every logical step. This capability supplies high-quality training signals for RL, thereby enabling models to develop stronger reasoning capabilities in rigorous settings.

Among interactive theorem provers, Lean \cite{DBLP:conf/cade/MouraKADR15,DBLP:conf/cade/Moura021} has cultivated a particularly robust ecosystem. Its unified mathematical library, \textit{mathlib4}, has expanded rapidly through large-scale community collaboration and, as of December 2025, contains over 250,000 theorems and 120,000 definitions. A landmark achievement in this domain is the \textit{Liquid Tensor Experiment}, led by Johan Commelin, which formalized a central theorem of Peter Scholze on liquid vector spaces. Scholze, who initially harbored doubts about the proof's correctness, remarked that the theorem might be his ``most important to date''\footnote{\href{https://xenaproject.wordpress.com/2020/12/05/liquid-tensor-experiment/}{https://xenaproject.wordpress.com/2020/12/05/liquid-tensor-experiment/}}. The project, which spanned approximately 18 months, not only verified the result but also led to a simplification of the original Clausen--Scholze proof, helping Scholze deepen his understanding of the argument's structure\footnote{\href{https://xenaproject.wordpress.com/2021/06/05/half-a-year-of-the-liquid-tensor-experiment-amazing-developments/}{Xena Project blog post, June 2021}}. Furthermore, the Liquid Tensor Experiment catalyzed the development of algebraic infrastructure within mathlib4: it drove the early formalization of homological algebra and category theory and attracted a wave of algebraists to the community. Other notable milestones include the formalization of the Polynomial Freiman--Ruzsa (PFR) conjecture\footnote{\href{https://teorth.github.io/pfr//}{https://teorth.github.io/pfr//}}, the sphere eversion theorem\footnote{\href{https://leanprover-community.github.io/sphere-eversion/}{https://leanprover-community.github.io/sphere-eversion/}}, ongoing work on formalizing the solution to the sphere packing problem in dimension 8\footnote{\href{https://thefundamentaltheor3m.github.io/Sphere-Packing-Lean/}{https://thefundamentaltheor3m.github.io/Sphere-Packing-Lean/}}, and the long-term effort to formalize Fermat’s Last Theorem\footnote{\href{https://leanprover-community.github.io/blog/posts/FLT-announcement/}{https://leanprover-community.github.io/blog/posts/FLT-announcement/}} lead by Kevin Buzzard. Expanding into applied mathematics, recent work has also established the formal foundations for numerical optimization in Lean4, specifically verifying the convergence of first-order algorithms \cite{li2025formalization}.

However, these projects remain labor-intensive, requiring experts to manually translate definitions and proofs into code. This high cost motivates the development of \textbf{autoformalization} tools and \textbf{automated theorem provers} to accelerate the \textit{digitalization} of mathematical knowledge—the conversion of standard informal mathematics into rigorous formal systems such as Lean.

\subsubsection{Autoformalization}

Autoformalization is the task of translating mathematical statements and proofs from natural language into formal code in an autonomous manner (e.g. by a language model). Early efforts in this domain employed sequence-to-sequence models trained on aligned data. For instance, \cite{wang2018first} constructed datasets by informalizing Mizar statements to train translation models. Addressing the scarcity of aligned pairs, \cite{wang2020exploration} subsequently explored unsupervised approaches based on cycle-consistency losses, where models learned to reconstruct statements by translating them between informal and formal domains without explicit supervision.

The advent of LLMs fundamentally shifted this paradigm. Research demonstrated that off-the-shelf LLMs could generate reasonable formalizations via few-shot prompting \cite{wu2022autoformalization,gadgil2022towards,agrawal2022towards,azerbayev2023proofnet}. Crucially, \cite{wu2022autoformalization} observed an asymmetry: translating formal code to natural language (informalization) is significantly easier for models than the reverse. This insight catalyzed the creation of large-scale synthetic datasets, where LLMs are used to informalize massive formal libraries (such as mathlib4) to generate high-quality aligned pairs for training specialized autoformalizers \cite{jiang2023multilingual,lu2024process,gao2025herald,liu2025rethinking}.

Recent work focuses on improving the quality and grounding of these systems. \textit{Herald} \cite{gao2025herald} introduces a hierarchical informalization strategy that respects the dependency graph of the library mathlib. By translating declarations in topological order, it ensures that the natural language descriptions of prerequisite concepts are available when translating dependent theorems. Herald further augments its data via tactic-based state synthesis and achieves over 96\% statement autoformalization accuracy on the miniF2F validation set. To enhance grounding, \textit{RAutoformalizer} \cite{liu2025rethinking} employs retrieval mechanisms to anchor generated code in existing formal declarations. Addressing the challenge of missing concepts, common in research-level mathematics, \cite{wang2025aria} proposed \textit{Aria}, an LLM-based agentic system. More generally, an LLM-based agent refers to a system that operates through an explicit interaction loop with its environment: it maintains intermediate states, reasons over observations, performs multi-step planning, and selects actions accordingly. These actions may include invoking external tools such as semantic retrieval, symbolic reasoning modules, or code synthesis components, with feedback from the environment used to guide subsequent decisions. Such agentic designs enable the decomposition of complex tasks into structured subtasks and support iterative refinement beyond single-pass generation \cite{wang2024survey}. Within this framework, Aria decomposes informal statements into a dependency graph of concepts; if a concept is absent from the library mathlib, the agent recursively defines it from the bottom up using semantic search and synthesis, effectively handling the ``long tail'' of mathematical vocabulary.

\paragraph{Evaluation and Verification}
Evaluating statement autoformalization is non-trivial. Once a statement is correctly autoformalized, evaluating an autoformalized proof becomes straightforward due to the formal system, which makes statement autoformalization the primary difficulty from an evaluation perspective. While human expert review is the gold standard, it is unscalable. The challenge, therefore, lies in developing automated metrics for correctness.
\begin{itemize}
    \item \textbf{With Ground Truth:} When a reference formal statement exists, correctness can be assessed via logical equivalence rather than mere string matching. For example, \textit{BEq} \cite{liu2025rethinking} utilizes a neural theorem prover to check if the generated statement and the ground truth can be derived from one another. Similar equivalence-checking approaches are explored in \cite{liu2025assess,poiroux2025reliable}.

    \item \textbf{Without Ground Truth (Semantic Verification):} In the absence of a reference, one must verify \textit{semantic correctness}—whether the formal code faithfully captures the intent of the informal statement. A naive approach involves ``back-translation'', where an LLM translates the code back to English for comparison \cite{ying2024lean,gao2025herald}. However, this is prone to errors as LLMs may miss subtle logical discrepancies. To mitigate this, \cite{xuejun2025mathesis} proposes \textit{Mathesis}, a fine-grained evaluation framework. Mathesis decomposes statements into assumptions and conclusions, evaluates the alignment of each component separately, and aggregates these scores using a fuzzy integral to rigorously reject inconsistencies. To further assist verification, Aria \cite{wang2025aria} enriches the context by retrieving detailed metadata (types, values, informal descriptions) for every term in the formal statement, enabling more accurate semantic judgments.
\end{itemize}
Reliable verifiers are not only essential for evaluation but also serve as critical reward models for RL, creating a feedback loop that iteratively improves autoformalization performance \cite{xuejun2025mathesis,lu2024process,huang2025formarl}.

\textit{Note:} This section has focused on \textit{statement} autoformalization. The autoformalization of \textit{proofs}, which requires translating logical reasoning steps rather than just definitions, is inextricably linked to automated theorem proving. We therefore discuss proof autoformalization in the context of proof generation in the following section.

\subsubsection{Automated Theorem Proving}

\begin{figure}[tb]
\centering
\includegraphics[width=0.95\linewidth]{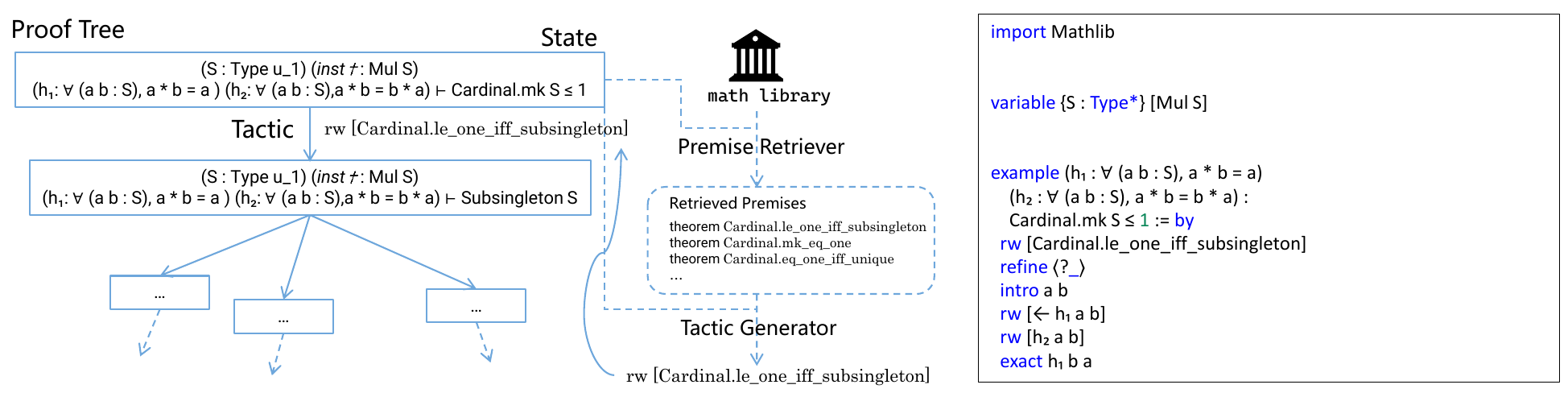}
\caption{Example of a proof step generation method and the complete proof it produces. Left: An illustration of a proof tree, where each node represents a proof state and each edge corresponds to a tactic. Applying an appropriate tactic transforms the current proof state into the next one. At each step, the proof step generation method retrieves relevant premises and proposes several candidate tactics to try. Right: The complete proof ultimately obtained by this process.}
\label{fig:proof-tree}
\end{figure}

Automated theorem proving in formal systems aims to generate valid proofs for formal statements. Deep learning--based approaches to this task fall into two broad categories: \textit{single-model approaches} and \textit{agentic approaches}. Single-model approaches can be further divided into \textit{proof step generation} and \textit{whole-proof generation} methods.

\paragraph{Proof Step Generation}
Proof step generation methods formulate theorem proving as a tree search problem. In this framework, each node in the search tree corresponds to a proof state, and each action corresponds to the application of a tactic that transitions the prover to a new proof state. A proof is successfully constructed once a path is found to a state with no remaining goals.  Figure \ref{fig:proof-tree} illustrates an example of a proof tree produced by such a method, alongside the final proof.

The strengths of proof step generation lie in \textit{reusability} and \textit{exploration}. During search, proof states are reusable: if a newly encountered state coincides with one previously explored, they can be merged. Furthermore, the system exhibits strong exploration capabilities by attempting multiple tactics at each step. However, these methods often suffer from slow inference speeds due to the computational cost of tree search, difficulty in training stability, and a heavy reliance on efficient interactive tools for both training and inference.

One of the earliest neural approaches in this domain is \textit{Holophrasm} \cite{whalen2016holophrasm}, which employs Monte Carlo Tree Search (MCTS) for exploration. Holophrasm integrates three neural components: a \textit{relevance network} for retrieving useful theorems, a \textit{generative network} for proposing variable substitutions, and a \textit{value network} for estimating provability. Subsequent work largely treated tactic prediction as a classification problem, with representative examples including \textit{GamePad} \cite{huang2018gamepad}, \textit{DeepHOL} \cite{bansal2019holist}, and graph-based approaches \cite{paliwal2020graph}. Moving beyond pure classification, \textit{GPT-f} \cite{polu2020generative} trains a Transformer to generate proof steps via a conditional language modeling objective, constructing proofs using best-first search. Similarly, \cite{lample2022hypertree} introduces a hypertree-based search combined with an online training strategy, where policy and critic models are periodically updated on data collected from repeated proof searches.

A major challenge in this domain is the scarcity of large-scale formal data. To mitigate this bottleneck, researchers have pursued two complementary directions: improving model performance under limited formal data, and synthesizing additional formal data. For the former, \textit{ABEL} \cite{gloeckle2024abel} proposes an online reinforcement learning framework with prioritized sampling to encourage exploration; notably, its RL stage is trained on only 244 formal statements yet achieves performance comparable to the state of the art at the time. For the latter, \textit{REAL-Prover} \cite{shen2025real} proposes an integrated pipeline comprising a statement autoformalizer to translate informal statements, a retrieval-augmented prover that generates tactics conditioned on relevant premises, and an \textit{expert iteration} paradigm. In this loop, the model trains on generated state-tactic pairs, performs proof search, and iteratively collects new training data from successful searches. Building on the retrieval-augmented architecture of REAL-Prover, the authors introduced the \textit{reap} \footnote{\href{https://github.com/frenzymath/reap}{https://github.com/frenzymath/reap}} tactic as a concrete instantiation of retrieval-guided proof automation. The reap tactic is currently in a beta-testing phase and maintained as a candidate feature branch for mathlib4.

A notable milestone is \textit{AlphaProof} \cite{hubert2025olympiad}.  AlphaProof trains a 3B-parameter proof network that jointly outputs tactics and value estimates. Its training pipeline involves pretraining on 300B tokens, supervised fine-tuning on 300k state-tactic pairs, and reinforcement learning on 80 million autoformalized statements. These formal statements are derived from approximately one million informal problems, where the autoformalization model is trained on triplets of (informal statement, formalization chain-of-thought, formal statement), with multiple distinct translations generated per problem. For particularly challenging tasks, AlphaProof applies \textit{test-time reinforcement learning} to adapt to the problem structure by constructing and training on a specialized curriculum. As a result, it achieves performance comparable to an IMO silver medalist. Other notable approaches include \cite{jiang2022thor, polu2023formal, wang-etal-2023-dt, yang2023leandojo, xin2025bfs, li2024hunyuanprover,xin2025scaling}.

\paragraph{Whole-Proof Generation}
In contrast, whole-proof generation methods aim to produce an entire formal proof, possibly augmented with inline comments, in a single forward pass. Their primary advantages are high inference speed and independence from interactive tools during generation. However, their exploration capability is limited compared to step-wise search; they often rely on behavior cloning and are more susceptible to error accumulation since they lack access to intermediate proof states.

This paradigm relies heavily on data quality and quantity. Since there is no principled a priori method for determining data quality, evaluation is often indirect via model performance. To address data quantity, \cite{xin2024deepseek} proposed an integrated pipeline involving automatic statement formalization, filtering (to remove trivial or false statements), statement proving, and iterative training on the resulting verified pairs. Building on this, \textit{DeepSeek-Prover-V1.5} \cite{xin2025deepseekproverv} improved performance by constructing richer datasets that include informal proofs written before formal code and inline informal comments, and by applying \textit{Reinforcement Learning from Verifier Feedback (RLVF)}. Other works adopting this paradigm include \cite{first2023baldur, azerbayev2024llemma, ying2024internlm, dong2025stp, zhang2025leanabell,wang2025kimina,lin2025goedel}.

\paragraph{Agentic Approaches}
Agentic approaches \cite{jiang2023draft,wang2024legoprover,varambally2025hilbert,Chen_et_al_2025_Seed-Prover1.5,jiang2025fate,gloeckle2025reinforcement,achim2025aristotle} represent a paradigm shift from single-model systems to modular workflows. These methods decompose theorem proving into coordinated sub-tasks, such as retrieval, decomposition, and verification, orchestrated through structured workflows that integrate language models with external tools.  The effectiveness of these approaches relies on three core components: robust retrieval systems, the reasoning capabilities of LLMs, and workflows designed to mimic the mathematical research process.

\textit{Draft, Sketch, and Prove (DSP)} \cite{jiang2023draft} is a prototype of this paradigm. It generates an informal proof, translates it into a formal sketch with open subgoals, and closes them using a lightweight prover. \textit{LEGO-Prover} \cite{wang2024legoprover} extends this by maintaining a persistent lemma pool. Uniquely, it evolves verified lemmas into new ones through strategies such as \textit{dimension extension}, \textit{key concept identification}, \textit{parameterization}, and \textit{complexity enhancement}. \textit{Hilbert} \cite{varambally2025hilbert} transforms informal proofs into formal sketches via recursive subgoal decomposition, guided by theorem retrieval. \textit{Seed-Prover-1.5} \cite{Chen_et_al_2025_Seed-Prover1.5} similarly employs a specialized sketch model and a dedicated prover model, achieving 80\% accuracy on the graduate-level benchmark FATE-H and 33\% accuracy on the PhD-level benchmark FATE-X \cite{jiang2025fate}. Notably, for agentic approaches such as Seed-Prover, the widely used miniF2F benchmark \cite{zheng2021minif2f} has largely been saturated and can be regarded as effectively solved. Addressing the granularity gap between informal reasoning and formal code, \cite{wang2025translating} proposes a two-stage Chain of States (CoS) framework. By extracting a sequence of intermediate formal states aligned with the logical flow of the informal argument before generating the specific transition tactics, this approach significantly reduces the complexity of tactic generation under constrained computational budgets.

Advanced agents like \textit{Aristotle} \cite{achim2025aristotle} interleave informal reasoning with formal verification. Aristotle drafts proofs as sequences of lemmas, formalizes them, and attempts verification, refining its outputs iteratively based on feedback. Together with a geometry solver, it achieved performance at the level of an IMO gold medal. \textit{Numina-Lean-Agent} \cite{liu2026numina} equips the coding agent Claude Code with Lean-LSP-MCP \cite{lean-lsp-mcp} for interacting with Lean, a semantic search engine for theorem retrieval, an informal prover \cite{huang2025winning}, and an external LLM for additional reasoning support. It successfully solved all problems in the 2025 Putnam Competition, an achievement first reported by \textit{AxiomProver} \cite{axiommath_fromSeeingWhy_2025}. Numina-Lean-Agent also formalized a research paper on Brascamp–Lieb inequalities \cite{benard2025effective} by interacting with mathematicians through iterative refinement of a natural language blueprint, represented as a dependency graph for both statements and proofs.
Finally, the \textit{Gauss} agent \cite{mathinc_gauss} demonstrated the power of human-AI collaboration by formalizing the strong Prime Number Theorem within three weeks, leveraging expert scaffolding. These results demonstrate that carefully designed agentic workflows can effectively exploit both intrinsic reasoning capabilities and external tools to achieve substantial gains in automated theorem proving.


The practical impact of formal reasoning agents is perhaps most clearly illustrated by their emerging role in recent work on conjectures posed by Paul Erdős, as discussed by Terence Tao \footnote{\href{https://github.com/teorth/erdosproblems/wiki/AI-contributions-to-Erd\%C5\%91s-problems}
     {https://github.com/teorth/erdosproblems/wiki/AI-contributions-to-Erdős-problems}}. During late 2025 and early 2026, agentic systems such as Aristotle became integral components of a high-leverage verification workflow: frontier language models (e.g., Gemini Deep Think, ChatGPT) were used to generate candidate constructions or heuristic arguments, which were then rigorously formalized and mechanically checked in Lean. This pipeline has already yielded formally verified counterexamples to several Erdős problems (notably Problem \#205) and partial resolutions of more intricate problems (e.g., \#367), while many related cases remain open. Importantly, these examples are best viewed as representative rather than definitive: the process is ongoing, and further results continue to emerge. While formalization dramatically reduces the risk of irreparable hallucinations, it does not eliminate all failure modes—models may still exploit unintended problem specifications or rediscover arguments later found in the literature. Nevertheless, these developments mark a qualitative shift. Formal reasoning agents now function as reliable research copilots, systematically converting heuristic insight into certified mathematical knowledge and mitigating long-standing verification bottlenecks in exploratory mathematics.

\subsection{Mathematical Information Retrieval}\label{subsec:mir}

Mathematical Information Retrieval concerns the task of retrieving mathematical content, including formulas, theorems, and problem solutions, from large collections of mathematical documents. In contrast to standard text retrieval, MIR must explicitly account for the distinctive structure and semantics of mathematical expressions. Mathematical formulas are inherently structured objects whose meaning depends on symbolic composition and relational structure rather than mere lexical overlap. Consequently, an effective MIR system must address challenges such as matching mathematical structure and symbolic patterns, while simultaneously leveraging surrounding textual context to resolve ambiguity and interpret meaning.

Crucially, MIR is not merely a search tool for human users but a foundational component of modern automated theorem proving (ATP) and AI agent systems. In the context of ATP, \textit{premise retrieval}, the task of identifying useful theorems, lemmas, or definitions from a vast library to prove a new theorem, is often the primary bottleneck. As mathematical libraries grow to contain hundreds of thousands of formal statements (e.g., mathlib4), the ability to efficiently retrieve the ``needle in the haystack'' determines whether a prover can solve a problem or times out. For agentic systems, MIR enables access to long-term mathematical memory, allowing agents to ground their reasoning in established knowledge rather than hallucinating unsupported facts. This requires a shift from traditional keyword-based matching to reasoning-based retrieval. A robust MIR model must understand logical implication and mathematical equivalence; for instance, it must recognize that a theorem stating ``a square matrix has a non-zero determinant'' is a critical premise required to answer a query about whether ``the columns of the matrix form a linearly independent set,'' despite the absence of shared keywords between the two statements.

Depending on the granularity of the retrieval target and the nature of the query, MIR encompasses several closely related tasks. Prominent examples include \textit{premise retrieval}, \textit{semantic retrieval}, and \textit{question-answer retrieval}.

\paragraph{Premise Retrieval}

In automated theorem proving, a central subproblem is premise retrieval: given a conjecture and a large library of existing mathematical statements, the system must identify which premises are likely to be useful for constructing a proof. Early approaches relied primarily on handcrafted similarity measures and heuristics\cite{hoder2011sine,meng2009lightweight}. Variants and extensions of these ideas, including tree-based similarity scores, have continued to be explored in more recent work \cite{wang2025treebased}. In parallel, lightweight machine learning methods such as k-nearest neighbors or sparse naive Bayes were also investigated for premise selection \cite{czajka2018hammer}.

Over the past decade, deep learning methods have increasingly been applied to premise retrieval. A representative early neural approach is \textit{DeepMath} \cite{irving2016deepmath}, which encodes conjectures and candidate premises separately and the resulting representations are concatenated and passed through a fully connected network that predicts whether a premise is useful for proving the conjecture. Training is supervised using existing proofs, where premises appearing in proofs are treated as positive examples and hard negative mining is employed to construct informative negative samples. Subsequent work sought to better exploit the internal structure of logical formulas. \textit{FormulaNet} \cite{wang2017premise}, for instance, represents each formula as a graph derived from its syntactic structure, with nodes corresponding to constants, variables, or quantifiers. A graph neural network is then used to compute embeddings, which are combined and fed into a classifier to estimate relevance.

Beyond pairwise scoring models, later work explored graph-level representations over entire libraries of statements. In \cite{ferreira-freitas-2020-premise}, the authors construct a global graph in which nodes correspond to mathematical statements and directed edges encode premise–conclusion relationships extracted from proofs. Premise selection for a new conjecture is formulated as a link prediction problem, and a graph convolutional network is used to score potential edges based on textual and structural features of nodes. In parallel, another line of research adopts embedding-based retrieval methods, treating each mathematical statement as text and encoding it into a single vector using a learned embedding model. Relevance is then assessed via similarity in the embedding space, often followed by a learned re-ranking stage to refine the retrieved candidate set. Training typically relies on contrastive objectives that draw conjectures closer to premises appearing in their proofs while pushing them away from irrelevant statements. Representative examples of this approach include \cite{yang2023leandojo,tao2025learning,shen2025real}.

\paragraph{Semantic Retrieval} 

Semantic retrieval aims to identify mathematically equivalent or closely related statements from a mathematical corpus based on their meaning rather than surface-level similarity. This task is motivated by practical use cases such as theorem search in large mathematical libraries. For example, users of Lean frequently need to locate relevant theorems in mathlib4 when constructing proofs. In this setting, a query may be expressed either in natural language or in formal code, while the retrieval corpus typically consists of formal declarations from mathlib4. To bridge the gap between informal queries and formal corpora, \textit{LeanSearch} \footnote{\href{https://leansearch.net/}{http://leansearch.net/}} constructs an aligned informal–formal corpus derived from mathlib4 and performs retrieval over the combined representation \cite{gao-etal-2024-semantic-search}. This approach enables semantic matching across representation modalities and significantly improves retrieval effectiveness for natural language queries. At present, LeanSearch is officially supported by the Lean community and has been integrated into mathlib4 as a standard client for semantic retrieval. This level of community endorsement and infrastructural integration makes LeanSearch a practical and sustainable component of the mathlib ecosystem, rather than a standalone research prototype. In addition to LeanSearch, several other semantic search tools have been developed for mathlib4, including \textit{Moogle} \footnote{\href{https://www.moogle.ai/}{https://www.moogle.ai/}}, \textit{LeanExplore} \cite{asher2025leanexplore}, \textit{LeanFinder} \cite{lu2025lean}, and \textit{LeanDex} \footnote{\href{https://leandex.projectnumina.ai/}{https://leandex.projectnumina.ai/}}. 

Beyond formal theorem retrieval, natural-language statement retrieval is also important for mathematical research. It can help mathematicians find related theorems, determine whether a result is already known, avoid redundant effort, and trace the historical origins of a theorem. Existing systems for natural-language semantic search over mathematical statements include \textit{TheoremSearch} \cite{alexander2026semantic}, which supports semantic search over theorem statements from sources such as arXiv papers, ProofWiki, the Stacks Project, and related mathematical corpora. In \textit{TheoremSearch}, theorem slogans serve as compact textual representations for embedding. Another example is \textit{Matlas} \cite{ju2026matlas}, a search engine over 435K published papers and 1.9K textbooks. In \textit{Matlas}, mathematical statements are extracted together with their dependency relations. For each document, these relations are used to construct a directed dependency graph; statements are then unfolded in topological order to obtain more self-contained representations, which are used for embedding and retrieval.

Formula retrieval constitutes an important subtask of semantic retrieval, where the query is a mathematical formula or a formula pattern and the objective is to retrieve semantically related formulas from a collection of documents. This setting introduces unique challenges. Formulas that represent the same mathematical concept may differ significantly in surface form due to notational variation or algebraic properties such as commutativity. Conversely, formulas with similar visual appearance may encode different meanings when interpreted in different mathematical contexts. Traditional approaches to formula retrieval are largely based on tree representations that encode the structural organization of mathematical expressions. In these methods, a formula is represented as a tree, and similarity is defined over trees or their substructures. A widely used representation is the \textit{Symbol Layout Tree (SLT)} \cite{DBLP:journals/ijdar/ZanibbiB12}, in which nodes correspond to symbols and edges encode spatial relationships such as superscript, subscript, or adjacency. Another prominent representation is the \textit{Operator Tree (OPT)} \cite{DBLP:conf/ntcir/GaoYWJT16}, where internal nodes represent operators and leaf nodes represent operands. Compared with SLTs, operator trees abstract away visual layout and focus on mathematical operations and their hierarchical relationships. Tree-based retrieval algorithms typically compare formulas by matching subtrees or paths, or by computing tree edit distances. For example, \textit{Approach0} \cite{DBLP:conf/ecir/ZhongZ19,10.1007/978-3-030-45439-5_47} represents formulas as operator trees and uses leaf-to-root paths as basic retrieval units. It performs retrieval in two stages, first identifying candidate formulas whose paths overlap with those of the query, and then re-ranking candidates based on similarity measures derived from the largest common subtrees. 

Beyond traditional symbolic matching, recent work has explored the use of text embedding models for formula retrieval. Early approaches adapt techniques from natural language representation by linearizing structured formula encodings and embedding them in a continuous vector space. For instance, \textit{Tangent-CFT} \cite{10.1145/3341981.3344235} performs a depth-first traversal over SLT and OPT, tokenizes the resulting tuple sequences, and applies text embedding models to obtain formula representations. Concurrent work augments formula representations with surrounding textual context to better capture semantic meaning \cite{10.1145/3511808.3557567,krstovski2018equation}. For example, \textit{MathAMR} \cite{10.1145/3511808.3557567} integrates formulas into their linguistic context by combining Abstract Meaning Representation (AMR) graphs with OPTs, replacing formula nodes in the AMR graph by the root nodes of the corresponding OPTs, and embedding the resulting linearized graphs using Sentence-BERT.

\paragraph{Question-Answer Retrieval}

Question–answer (QA) retrieval concerns the task of retrieving mathematical answers, explanations, or supporting documents in response to a natural language query. Mathematical questions are inherently multimodal, typically combining natural language with symbolic expressions, formulas, or diagrams, and candidate answers exhibit similar structure. As a result, relevance in mathematical QA retrieval is defined by semantic adequacy, namely whether an answer correctly and meaningfully addresses the question, for example by providing a valid solution, proof, or conceptual explanation, rather than by surface-level lexical overlap.

Traditional general-purpose text retrieval methods, such as BM25 \cite{robertson1995okapi,10.1561/1500000019} and TF–IDF, rely on sparse vector representations of queries and documents and assess relevance through weighted vector similarity \cite{SALTON1988513}. While these methods can be directly applied to question–answer retrieval, they often perform poorly in mathematical settings. This limitation stems from their dependence on exact term matching and their inability to capture the semantics of mathematical language or the structural relationships encoded in mathematical formulas. With the rise of deep learning, research shifted toward neural retrieval models based on pretrained transformers. A common approach is to pretrain and fine-tune transformer models on large-scale mathematical corpora to obtain representations better aligned with mathematical syntax and semantics. For example, \textit{MathBERT} \cite{peng2021mathbert} was pretrained on math-heavy corpora containing formulas and introduced objectives such as masked formula substructure prediction to better model mathematical symbols in context. Building on dense retrieval paradigms, \cite{reusch2021tu_dbs} investigated the use of \textit{ColBERT} \cite{khattab2020colbert} on the \textit{ARQMath} benchmarks \cite{10.1007/978-3-030-85251-1_17,10.1007/978-3-030-58219-7_15}, fine-tuning neural retrievers on millions of question–answer pairs with negatives selected via rule-based heuristics. Recognizing the complementary strengths of symbolic and neural methods, several hybrid approaches have also been proposed. \textit{Mabowdor} \cite{10.1145/3539618.3591746}, for example, combines dense passage retrieval with parallel sparse retrieval based on structure-aware mathematical indexing and traditional text-based methods, merging their outputs via a learned weighted scheme. This hybrid strategy achieved strong performance in ARQMath-3 \cite{Mansouri2022OverviewOA}, highlighting the effectiveness of integrating classical mathematical structure with neural semantic representations for QA retrieval.

\section{Challenges and Prospects}\label{sec:cp}

Despite the promising progress in AI for Mathematics, the field faces a fundamental hurdle: current AI systems, particularly foundation models, still lack the depth of reasoning required for research-level mathematics. Bridging this gap necessitates a shift from passive assistance to active learning within a rigorous ``logic environment.'' This requires accelerating the formalization, or digitalization, of mathematics to provide automated, verifiable feedback that can iteratively strengthen AI reasoning. Furthermore, advancing these capabilities demands the infusion of professional mathematical knowledge, ranging from high-quality data curation to the design of specialized agentic workflows, into the model development process. Ultimately, the goal is to integrate AI seamlessly into the daily practice of mathematicians, a vision that can only be realized through continuous, close collaboration between AI researchers, engineers, and the mathematical community. We summarize these key challenges and future directions below:

\begin{enumerate}
    \item \textbf{Domain Expertise and Feature Engineering:} In problem-specific modeling, the design of input features often necessitates deep domain expertise. Human intuition remains indispensable, both for selecting mathematically meaningful features and for interpreting model outputs to derive theoretical insights. This dependency holds equally for discovery-oriented agents (e.g., AlphaEvolve-style systems), which rely on carefully handcrafted representations and scoring functions. Consequently, the development of effective AI for Mathematics requires close and sustained collaboration between machine learning researchers and domain experts to ensure that computational outcomes translate into genuine mathematical progress.

    \item \textbf{The Verification Bottleneck and Autoformalization:} Accurate and efficient verification constitutes a critical bottleneck in research-level mathematics. The inherent ambiguity of natural language, combined with the scarcity of experts capable of auditing advanced proofs, makes manual verification slow and error-prone. Grounding mathematical reasoning in formal languages, where correctness can be mechanically verified, is often proposed as a means to ensure the reliability of verification. However, the formal reasoning capabilities of current LLMs lag behind their natural language performance due to the severe scarcity of high-quality formal data. Addressing this ``formal data gap'' requires the development of robust autoformalization tools that bridge informal and formal mathematics. By building sound infrastructure for specific subfields and supporting repository-level formalization, we can accelerate the translation of natural language reasoning into formal proofs. This creates a virtuous cycle: formally verifiable feedback can serve as a high-quality training signal to further enhance the reasoning capabilities of LLMs in both mathematics and broader domains. Nevertheless, it is important to recognize that informal verification continues to play a valuable role in the verification of research-level mathematics. The \textit{Aletheia} \cite{feng2026towards} math research agent, for instance, demonstrates how informal verification, through a combination of natural language generation, verification, and revision, can autonomously address mathematical research problems and produce publishable results, highlighting the ongoing relevance of informal verification within the field.

    \item \textbf{Semantic Consistency in Formalization:} A subtle challenge in autoformalization is verifying the \emph{semantic correctness} of generated formal statements, namely whether the formalized statement faithfully reflects the intent of the original informal statement. Existing models often struggle to determine whether a back-translated formal statement faithfully captures the nuance of the original informal statement. This necessitates the development of fine-grained, robust semantic-consistency verifiers. While the final judgment of semantic intent must arguably remain with human experts to ensure conceptual accuracy, automated systems can serve as efficient first-stage filters. By reducing the volume of candidates that require manual review, such systems can scale the formalization process without compromising rigorous standards.

    \item \textbf{Beyond Correctness to Understanding:} Formal validity is a necessary, but not sufficient, condition for mathematical value. As William Thurston famously noted \cite{c7da7d34-de5b-3a8d-8096-59d6b94e9a82}, ``mathematics is not about numbers, equations, computations, or algorithms: it is about understanding.'' A valuable proof does more than establish truth; it provides insight, reveals structure, and offers techniques applicable to other problems. Similarly, Stanislaw Ulam \cite{ulam1976adventures} quotes Stefan Banach as saying, ``Good mathematicians see analogies between theorems or theories, the very best ones see analogies between analogies.'' This highlights a broader truth: the quality of a proof lies in its ability to deepen our conceptual grasp of the mathematical landscape. Future AI systems must therefore aim to go beyond mere verification, assisting in the discovery of proofs that reshape our thinking and uncover connections that were previously invisible. 

    \item \textbf{From Heuristics to Expert Routines:} While standalone LLMs serve as powerful reasoning engines, the future of AI4Math lies in agentic systems designed to emulate the sophisticated workflows of professional mathematicians. Research-level mathematics is rarely a linear deduction; it involves a complex, iterative cycle of constructing examples, consulting literature, formulating conjectures, and refining proof strategies based on intermediate failures. Current agents, however, remain largely generic. A critical frontier is to develop architectures that explicitly model these expert ``routines'', learning to orchestrate tools and strategies in a way that mirrors the cognitive processes of a researcher. This includes the use of ``computational sketching''—leveraging code generation not just for formal proofs, but to construct numerical toy-examples or perform symbolic derivations that quickly verify or falsify human intuition. Furthermore, these agents can automate high-value long-tail tasks often neglected by humans, such as proof reorganization, condition weakening, and the semantic retrieval of obscure existing solutions. Ultimately, the goal is not merely to mimic human workflows but to optimize them, creating agents capable of navigating the search space of mathematical ideas with a systematicity and scale that surpasses human capability in attacking challenging problems.

    \item \textbf{Active Community Participation:} Resonating with the necessity of domain expertise, the advancement of AI reasoning requires the active intervention of mathematicians. Beyond generating formal data, the community must actively explore these systems to develop a collective ``mental model'' of their capabilities and boundaries. For instance, distinguishing whether models are more adept at algebraic manipulation than geometric topology is crucial for determining where AI can be reliably deployed. This involves not only accelerating the \textit{digitalization} of mathematical knowledge to create verifiable training corpora but also engaging in ``adversarial collaboration'' to identify logical gaps. By rigorously mapping these strengths and weaknesses, mathematicians can guide the development of models that are not just statistically powerful but mathematically sound.
    
    \item \textbf{Embracing AI-Assisted Research:} We must prepare for a cultural shift where AI evolves from a computational tool into a research \textit{copilot}. Developments in late 2025, highlighted by Georgiev, G{\'o}mez-Serrano, Tao and Wagner's work \cite{georgiev2025mathematical}, underscore this transition. In this work, the authors observed that while these models may still lack true understanding, often ``doing impressions of thinking'', they have become adept at autonomously discovering mathematical constructions that elude human intuition. Even when models hallucinate or produce flawed reasoning, their ability to generate plausible structural candidates allows them to act as effective co-pilots, guiding researchers toward fruitful avenues of inquiry while leaving the final, rigorous verification to human experts.
    
    We would argue that even if an AI's divergent reasoning, proposing random or creative variants, has a low probability of correctness, the total research efficiency improves provided the ``verification leverage'' is high. In many advanced mathematical domains, generating a solution is computationally or cognitively expensive, whereas verifying a candidate solution is comparatively rapid. This asymmetry allows researchers to utilize AI as a high-volume generator of candidate ideas, where the time saved by a single valid insight significantly outweighs the low cost of discarding incorrect suggestions.
    
    However, realizing this potential requires more than just powerful models; it demands well-engineered, accessible tools. To facilitate high engagement, the barrier to entry must be lowered through robust software design. The ease of use seen in recent frameworks like AlphaEvolve, compared to earlier prototypes, demonstrates that engineering quality is a critical determinant in transitioning these technologies from experimental curiosities to standard instruments of global adoption.
    
\end{enumerate}

\section*{Acknowledgements}

The authors would like to thank Gergely Bérczi, Xiaoshan Gao, Amaury Hayat, Xuhua He, Kyu-Hwan Lee, Geordie Williamson, and Wotao Yin for valuable comments and suggestions that helped improve the presentation of this paper. They also thank Guoxiong Gao and Chenyi Li for collecting the PKU Undergraduate Final Exams and PhD Qualifying Exams. The authors further acknowledge Feiyue Ye, Yining Wang, Rui Wang, Youle Fang and Wenhao Xu for their careful assessment and valuable feedback on machine-generated solutions for the PKU Undergraduate Final Exams, as well as Qing Lan, Zekun Chen, Shuzhe Cai, Shengbo Dong and Yutong Wang for their assessment of solutions for the PKU PhD Qualifying Exams. This work is supported in part by the National Key R\&D Program of China grant 2024YFA1014000, the Fundamental and Interdisciplinary Disciplines Breakthrough Plan of the Ministry of Education of China (JYB2025XDXM113), the New Cornerstone Investigator Program, and the project ``Planning Research and Formulation on Artificial Intelligence for Science for the 15th Five-Year Plan," commissioned by the Fourth Department of the Ministry of Science and Technology of the People's Republic of China.


\bibliographystyle{plain}
\bibliography{ref}

\newpage
\appendix

\section{Scoring Criteria}

\begin{table}[h]
\label{tab:score_crtr}
\centering
\resizebox{0.9\linewidth}{!}{
\begin{tabular}{cl}
\toprule
\textbf{Score} & \textbf{Criteria} \\
\midrule
0 & The approach is entirely incorrect. \\
1 & The reasoning is confused and the errors are fatal, but the first one or two lines head in the right direction.\\
2 & The main idea is correct, but errors in the middle cause the entire latter part to be wrong. \\
3 & The main idea is correct; although there are mistakes, the proof eventually returns to the right track and the latter part is correct.\\
4 & The overall proof is correct, with only minor, non-essential errors.\\
5 & Completely correct.\\
\bottomrule
\end{tabular}}
\caption{Scoring criteria for the undergraduate exams.}
\end{table}

\section{Example Problems from PKU Undergraduate Final Exams}\label{sec:exam}

\begin{thm}
For integers $1 < p < q$, define $T_{p, q}(x)$ as the following triangular polynomial: $\frac{\cos((q-p)x)}{p} + \frac{\cos((q-p+1)x)}{p-1} + \cdots + \frac{\cos((q-1)x)}{1} - \frac{\cos((q+1)x)}{1} - \frac{\cos((q+2)x)}{2} - \cdots - \frac{\cos((q+p)x)}{p}$ For $k \geq 1$, choose integers $1 < p_k < q_k$ and real numbers $a_k > 0$ such that $q_k + p_k < q_{k+1} - p_{k+1}$, $\sum_{k=1}^{\infty} a_k < +\infty$, and $\liminf_{k \to \infty} a_k \ln p_k > 0$. (Numbers satisfying the above requirements exist, such as $p_k = 2^{k^3}$, $q_k = 2^{k^3 + 1}$, $a_k = \frac{1}{k^2}$) (1) Prove that $f(x) = \sum_{k=1}^{\infty} a_k T_{p_k, q_k}(x)$ is a continuous function with period $2\pi$. (2) Does the Fourier series of $f(x)$ converge at $x = 0$? Provide a justification.
\end{thm}

\begin{proof}[Solution (GPT-4)]
 To tackle this problem, we'll break it down into two main parts as requested: proving the continuity and periodicity of \(f(x)\), and then discussing the convergence of its Fourier series at \(x = 0\).
  
  Part 1: Continuity and Periodicity of \(f(x)\)
  
  Continuity
  
  To prove \(f(x)\) is continuous, we need to show that the series of functions \(\sum_{k=1}^{\infty} a_k T_{p_k, q_k}(x)\) converges uniformly on any closed interval in \(\mathbb{R}\), because a uniformly convergent series of continuous functions is continuous.
  
  Each \(T_{p_k, q_k}(x)\) is a finite sum of continuous functions (\(\cos(nx)\) terms divided by integers), and hence is continuous. The uniform convergence follows from the Weierstrass M-test if we can show that \(\sum_{k=1}^{\infty} M_k < \infty\), where \(M_k\) is a bound for \(|a_k T_{p_k, q_k}(x)|\) for each \(k\).
  
  For any \(x\), the maximum value of \(|\cos(nx)|\) is 1. Thus, we can bound each term of \(T_{p_k, q_k}(x)\) by its denominator's reciprocal. The sum of these bounds forms a telescoping series, which simplifies to a constant times \(p_k\). Therefore, we can choose \(M_k = a_k C p_k\) for some constant \(C\) that depends on the structure of \(T_{p_k, q_k}(x)\) but not on \(k\). Given \(\sum_{k=1}^{\infty} a_k < +\infty\) and \(p_k\) grows, but not faster than the decrease of \(a_k\), the series \(\sum_{k=1}^{\infty} M_k\) converges. Hence, by the Weierstrass M-test, \(f(x)\) is continuous.
  
 Periodicity
  
  The periodicity of \(f(x)\) with period \(2\pi\) follows directly from the periodicity of the cosine function, which has a period of \(2\pi\). Since each \(T_{p_k, q_k}(x)\) is a sum of cosine functions with integer multiples of \(x\), each \(T_{p_k, q_k}(x)\) is periodic with period \(2\pi\). Therefore, the sum \(f(x)\), being a sum of \(2\pi\)-periodic functions, is also \(2\pi\)-periodic.
  
  Part 2: Convergence of the Fourier Series at \(x = 0\)
  
  The Fourier series of a \(2\pi\)-periodic function \(f(x)\) is given by the sum of its Fourier coefficients times the corresponding cosine and sine functions. At \(x = 0\), the sine terms vanish, and we are left with the sum of the Fourier cosine coefficients.
  
  To discuss the convergence of the Fourier series of \(f(x)\) at \(x = 0\), we need to consider the behavior of its Fourier coefficients. The Fourier coefficients involve integrals of \(f(x)\) times sine or cosine functions over a period. Given that \(f(x)\) is defined as a sum of functions involving cosines, the Fourier coefficients will reflect the structure of these cosines.
  
  However, the question of convergence at \(x = 0\) simplifies to considering the sum of the constant terms in the Fourier expansion of \(f(x)\). Since \(f(x)\) is continuous and periodic, its Fourier series converges to \(f(x)\) at points where \(f(x)\) is continuous, including \(x = 0\), by Dirichlet's test for the convergence of Fourier series.
  
  In summary, the Fourier series of \(f(x)\) does converge at \(x = 0\), and it converges to \(f(0)\) due to the continuity of \(f(x)\) at that point. This conclusion is supported by the general theory of Fourier series, which guarantees convergence at points of continuity.
\end{proof}

\begin{proof}[Solution (o1-2024-12-17)]

We wish to construct and study the function  
\[
f(x)=\sum_{k=1}^{\infty} a_k\, T_{p_k,q_k}(x),
\]
where each \(T_{p,q}(x)\) is the ``triangular polynomial''
\[
T_{p,q}(x)
= \frac{\cos((q-p)x)}{p} + \frac{\cos((q-p+1)x)}{p-1} + \cdots + \frac{\cos((q-1)x)}{1}
\]
\[
\qquad\qquad
- \Bigl[ \frac{\cos((q+1)x)}{1} + \frac{\cos((q+2)x)}{2} + \cdots + \frac{\cos((q+p)x)}{p} \Bigr],
\]
and where the integers \(p_k, q_k\) and the positive real coefficients \(a_k\) satisfy:

1) \(1 < p_k < q_k\) for each \(k\),  

2) \(q_k + p_k < q_{k+1} - p_{k+1}\)  (so the ``frequency blocks'' do not overlap),  

3) \(\sum_{k=1}^{\infty} a_k < +\infty\),  

4) \(\liminf_{k\to\infty} (a_k \ln p_k) > 0.\)

Part (1).  \(f(x)\) is continuous and \(2\pi\)-periodic.

\hrulefill

\textbf{Step 1A.  A simpler closed-form for \(T_{p,q}(x)\).}

Although \(T_{p,q}(x)\) was originally given as a sum of cosines with certain \(1/r\) coefficients, one sees more structure by rewriting it in the following way. Observe that

\[
T_{p,q}(x)
  = \sum_{r=0}^{p-1} \frac{\cos((q-p+r)x)}{p-r}
    \;-\;
    \sum_{r=1}^{p} \frac{\cos((q+r)x)}{r}.
\]

Make the index-shift \(n=q-m\) in the first sum and \(n=q+r\) in the second sum. A straightforward check shows that

\[
T_{p,q}(x)
= \sum_{m=1}^{p} \frac{\cos((q-m)x)}{m}
  - \sum_{r=1}^{p} \frac{\cos((q+r)x)}{r}.
\]

Using the trigonometric identity
\[
\cos(A)-\cos(B)=-2\sin\!\left(\frac{A+B}{2}\right)\sin\!\left(\frac{A-B}{2}\right),
\]
one finds
\[
\cos((q-m)x)-\cos((q+m)x)=2\sin(qx)\sin(mx).
\]

Hence
\[
T_{p,q}(x)
= \sum_{m=1}^{p} \frac{\cos((q-m)x)-\cos((q+m)x)}{m}
= 2\sin(qx)\sum_{m=1}^{p} \frac{\sin(mx)}{m}.
\]

Thus a much simpler, and very useful, identity is

\[
\tag{$\star$}
T_{p,q}(x)=2\sin(qx)\cdot S_p(x),
\qquad
S_p(x):=\sum_{m=1}^{p} \frac{\sin(mx)}{m}.
\]

Since sine is bounded by \(1\), one immediately gets

\[
|T_{p,q}(x)| \le 2\,|\sin(qx)|\,|S_p(x)|
\le 2\sum_{m=1}^{p} \frac{1}{m}.
\]

Recalling that the \(p\)-th harmonic number \(H_p = 1 + 1/2 + \cdots + 1/p\) is on the order of \(\ln p\), one has

\[
\tag{1}
|T_{p,q}(x)| \le 2H_p \approx 2(\ln p + \text{constants}).
\]

In particular, \(T_{p,q}(x)\) is a finite linear combination of cosines, hence is continuous and \(2\pi\)-periodic.

\hrulefill

\textbf{Step 1B.  Pointwise (and even uniform) convergence of \(f(x)\).}

We define  
\[
f(x)=\sum_{k=1}^{\infty} a_k\, T_{p_k,q_k}(x).
\]

Each \(T_{p_k,q_k}(x)\) is \(2\pi\)-periodic and continuous, so any uniform limit of partial sums would also be continuous and \(2\pi\)-periodic. One must check carefully that the series actually does converge (and is reasonably well-behaved) despite the factor \(\ln p_k\) in the upper bound (1).

The crucial extra condition is that the ``frequency blocks'' in \(T_{p_k,q_k}(x)\) do not overlap, because

\[
q_k + p_k < q_{k+1} - p_{k+1}.
\]

Indeed, \(T_{p_k,q_k}(x)\) uses cosines of integer frequencies \(n\) in the range
\[
n = q_k - p_k,\ldots, q_k + p_k,
\]
whereas \(T_{p_{k+1},q_{k+1}}(x)\) uses \(n\) in the next interval
\[
[q_{k+1}-p_{k+1},\ldots,q_{k+1}+p_{k+1}],
\]
which does not intersect the previous one. Consequently, when we form the partial sums

\[
F_n(x)=\sum_{k=1}^{n} a_k\, T_{p_k,q_k}(x),
\]

each \(F_n(x)\) is simply a finite linear combination of cosines with distinct integer frequencies. In particular, each \(F_n(x)\) is a trigonometric polynomial (hence continuous, \(2\pi\)-periodic), and two different blocks of frequencies do not ``interfere'' with each other.

One can then show (by a Dirichlet-type test in \(k\), or by exploiting the fact that the partial sums in \(x\) of those trigonometric blocks do not reinforce each other) that \(F_n(x)\) converges uniformly to a limit function \(f(x)\). The simplest high-level reason is:

\begin{itemize}
\item The series \(\sum a_k\) converges (so its partial sums are bounded).
\item Each \(T_{p_k,q_k}(x)\) oscillates in \(k\) because \(\sin(q_k x)\) in \((\star)\) has bounded partial sums in \(k\), etc.
\end{itemize}

In more elementary terms, one can check that for large indices \(m,n\), the difference

\[
\|F_n - F_m\|_{\infty}
= \sup_x \left| \sum_{k=m+1}^{n} a_k T_{p_k,q_k}(x) \right|
\]

stays small, thanks to the disjoint frequency support and the boundedness of those partial sums. Hence  
\[
\sum a_k T_{p_k,q_k}(x)
\]
converges uniformly to a continuous, \(2\pi\)-periodic limit \(f(x)\).  
This completes the proof of Part (1).

\hrulefill

\textbf{Part (2).  The Fourier series of \(f(x)\) does not converge at \(x=0\).}

Here we ask about the usual notion of ``Fourier series convergence at a point''—namely, taking the \(N\)-th partial sum of the full Fourier series (ordered by frequency \(n=0,1,2,\ldots\)), and then evaluating at \(x=0\). We claim that these partial sums do not converge to the actual value \(f(0)\).

\textbf{Step 2A.  Note that \(f(0)=0\).}

Indeed, because \(\cos(0)=1\), we get immediately from the original definition

\[
T_{p,q}(0)
= \left( \frac1p + \frac1{p-1} + \cdots + \frac11 \right)
-
\left( \frac11 + \frac12 + \cdots + \frac1p \right)
= 0.
\]

Thus for each \(k\),

\[
T_{p_k,q_k}(0)=0,
\]

and hence

\[
f(0)=\sum_{k=1}^{\infty} a_k T_{p_k,q_k}(0)=0.
\]

\textbf{Step 2B.  The partial sums of the Fourier series at \(0\) ``jump'' and fail to converge.}

Next, consider how the partial sums of the full Fourier expansion of \(f\) behave at \(x=0\). Since the frequency blocks
\[
[q_k - p_k,\ldots, q_k + p_k]
\]
do not overlap for different \(k\), each block’s contributions to the Fourier coefficients sum to zero:

\[
\sum_{n\in \text{block }k}
\bigl[\text{Fourier coefficient of }\cos(nx)\bigr]
= T_{p_k,q_k}(0) = 0.
\]

But the key point is that if we only include \emph{part} of the \(k\)-th block in the partial sum, we get a large partial harmonic sum that does not vanish. Concretely:

\begin{itemize}
\item If \(N\ge q_k + p_k\), then the entire block \(k\) is included, and its contribution is \(T_{p_k,q_k}(0)=0\).
\item If \(N\) is just before the end of block \(k\)—e.g. \(N=q_k - 1\) or \(N=q_k\)—then only the frequencies  
\(n=q_k - p_k,\ldots, q_k - 1\)  
appear. These coefficients add up to something of size \(a_k H_{p_k}\) (where \(H_{p_k}\sim \ln p_k\)).
\end{itemize}

More precisely, the partial Fourier sum up to \(N=q_k - 1\) includes the full blocks for all \(j<k\), but only the ``left half'' of block \(k\), yielding

\[
\text{partial sum at }0
= a_k \left( \frac11 + \frac12 + \cdots + \frac1{p_k} \right)
= a_k H_{p_k},
\]

which is about \(a_k \ln p_k\). By hypothesis,
\[
\liminf_{k\to\infty} a_k \ln p_k > 0,
\]
so this partial sum does not go to \(0\). Indeed, as \(N\) ``enters'' a new block, the partial sum at \(0\) jumps to about \(a_k \ln p_k\). Then, once \(N\) passes \(q_k + p_k\), that entire \(k\)-th block again sums to \(0\). Thus the partial sums at \(x=0\) keep oscillating between approximately \(a_k \ln p_k\) and \(0\), never settling down.

In particular, the sequence of partial sums at \(0\) fails to converge to \(f(0)=0\). Hence the Fourier series of \(f\) does not converge (in the usual sense) at \(x=0\).

\bigskip

\textbf{Answer Summary}

1) Because the blocks \(T_{p_k,q_k}(x)\) have disjoint frequency supports and the coefficients \(a_k\) sum to a finite number, one shows (by standard uniform-convergence tests for lacunary trigonometric series) that  
\(\sum a_k T_{p_k,q_k}(x)\)  
converges uniformly. Each triangular block is continuous and \(2\pi\)-periodic, so its uniform limit \(f(x)\) is also continuous and has period \(2\pi\).

2) Nonetheless, the usual partial sums of the Fourier series of \(f\) at \(x=0\) do not converge to \(f(0)\). Indeed \(f(0)=0\); but whenever the partial sums include only a ``partial block,'' one sees a large jump of size roughly \(a_k \ln p_k\), which does not vanish as \(k\) grows. Thus the classical Fourier partial sums fail to converge at \(x=0\).

\end{proof}

\textbf{Comment}: For GPT-4, the solution to the first question is correct, but the answer to the second question is not. In contrast, o1-2024-12-17 gives correct solutions to both questions, and its answer to the second question is particularly good because it observes that if a partial Fourier sum includes only part of the $k$-th frequency block, it produces a large partial harmonic sum that does not vanish.

\begin{thm}
    Let $\mathbb{R}^{n \times n}$ denote the set of all $n \times n$ real matrices, and $E_{ij}$ be the $n \times n$ matrix with $1$ at position $(i,j)$ and $0$ elsewhere. Define $\Gamma_r$ as the set of all $n \times n$ real matrices with rank $r$. Let $\phi : \mathbb{R}^{n \times n} \to \mathbb{R}^{n \times n}$ satisfy $\phi(AB)=\phi(A)\phi(B)$ for all $A,B \in \mathbb{R}^{n \times n}$. Prove:  
(1) $\forall A,B \in \Gamma_r$, $\operatorname{rank}(\phi(A))=\operatorname{rank}(\phi(B))$;  
(2) If $\phi(0)=0$ and there exists a rank $1$ matrix $W$ such that $\phi(W)\neq 0$, then there exists an invertible matrix $R$ such that $\phi(E_{ij}) = R E_{ij} R^{-1}$ holds for all $E_{ij}$, $i,j = 1,2,\ldots,n$.

\end{thm}

\begin{proof}[Solution (GPT-4)]
    To prove the given statements, we will proceed step by step, ensuring that each part of the proof is logically sound and well-explained.

\paragraph{Part 1: $ \forall A,B \in \Gamma_r$, $\operatorname{rank}(\varphi(A))= \operatorname{rank}(\varphi(B))$}

\textbf{Proof:}

Given that $\varphi:\mathbb{R}^{n\times n}\to\mathbb{R}^{n\times n}$ satisfies $\varphi(AB)=\varphi(A)\varphi(B)$ for all $A,B\in\mathbb{R}^{n\times n}$, we need to show that for any two matrices $A$ and $B$ with the same rank $r$, the rank of $\varphi(A)$ is equal to the rank of $\varphi(B)$.

\begin{enumerate}
  \item \textbf{Rank-nullity theorem:} Recall that for any linear transformation $T: V \to W$, the rank-nullity theorem states that $\dim(V) = \operatorname{rank}(T) + \operatorname{nullity}(T)$. This theorem will be foundational in understanding how transformations affect the rank of matrices.
  
  \item \textbf{Rank preservation under similar transformations:} If two matrices $A$ and $B$ are similar, i.e., there exists an invertible matrix $P$ such that $A = PBP^{-1}$, then they have the same rank. This is because similarity transformations do not affect the linearly independent rows or columns of a matrix.
  
  \item \textbf{Homomorphism property:} The given condition $\varphi(AB)=\varphi(A)\varphi(B)$ indicates that $\varphi$ is a homomorphism between the algebraic structures $(\mathbb{R}^{n\times n}, \cdot)$ under matrix multiplication. This property will be crucial in understanding how $\varphi$ preserves or changes the rank of matrices.
\end{enumerate}

To prove that $\operatorname{rank}(\varphi(A))= \operatorname{rank}(\varphi(B))$ for all $A,B\in\Gamma_r$, consider the following steps:

\begin{itemize}
  \item \textbf{Step 1:} Since $A$ and $B$ are in $\Gamma_r$, they have rank $r$. This means there exist matrices $U$ and $V$ such that $A = UDV$ and $B = XDY$, where $D$ is a diagonal matrix with $r$ non-zero entries (representing the rank), and $U, V, X, Y$ are invertible matrices. This decomposition is a form of the singular value decomposition (SVD), which exists for any real matrix.
  
  \item \textbf{Step 2:} Apply $\varphi$ to both sides of the equations $A = UDV$ and $B = XDY$. Using the homomorphism property, we get $\varphi(A) = \varphi(U)\varphi(D)\varphi(V)$ and $\varphi(B) = \varphi(X)\varphi(D)\varphi(Y)$.
  
  \item \textbf{Step 3:} Since $D$ has rank $r$, and both $A$ and $B$ are constructed using $D$ with invertible matrices, it's reasonable to infer (though not directly) that the transformation $\varphi$ applied to matrices of rank $r$ would yield matrices whose rank is determined by the transformation applied to $D$. However, without additional properties of $\varphi$ (e.g., linearity, injectivity), directly concluding that $\operatorname{rank}(\varphi(A)) = \operatorname{rank}(\varphi(B))$ requires further justification.
\end{itemize}

\textbf{Gap in Proof for Part 1:} The direct leap from the homomorphism property to rank preservation without additional properties of $\varphi$ (such as it being a linear transformation that preserves rank or injectivity/surjectivity) is a logical gap. To rigorously prove that $\operatorname{rank}(\varphi(A)) = \operatorname{rank}(\varphi(B))$ for all $A, B \in \Gamma_r$, we would need more information about the nature of $\varphi$ beyond its homomorphism property. For instance, if $\varphi$ were shown to be a linear transformation that either preserves rank or has a predictable effect on rank, we could make more definitive statements.

\paragraph{Part 2: If $\varphi(0)=0$ and there exists a rank 1 matrix $W$ such that $\varphi(W)\neq 0$, then there exists an invertible matrix $R$ such that $\varphi(E_{ij})=RE_{ij}R^{-1}$ holds for all $E_{ij}$, $i,j=1,2,\ldots,n$.}

\textbf{Proof:}

Given the conditions, we aim to show that $\varphi$ acts similarly to a conjugation by an invertible matrix $R$ on the basis matrices $E_{ij}$.

\begin{enumerate}
  \item \textbf{Existence of $W$:} The existence of a rank 1 matrix $W$ such that $\varphi(W) \neq 0$ suggests that $\varphi$ does not annihilate all rank 1 matrices. Since $E_{ij}$ are also rank 1 matrices, it's plausible to consider how $\varphi$ acts on these basis elements.
  
  \item \textbf{Conjugation property:} The statement we need to prove essentially says that $\varphi$ acts on $E_{ij}$ as if it were conjugating them with some fixed invertible matrix $R$. This would mean that $\varphi$ preserves the structure of $E_{ij}$ matrices up to a basis change represented by $R$.
\end{enumerate}

\textbf{Gap in Proof for Part 2:} To establish the existence of such an $R$, we would typically rely on the structure of $\varphi$ being such that it preserves linear combinations and acts uniformly on the basis elements of the space of matrices. However, without explicit details on how $\varphi$ acts on general matrices and without assuming linearity of $\varphi$, constructing such an $R$ directly from the given conditions is non-trivial.

\textbf{Conclusion:} The proofs for both parts as presented above contain logical gaps primarily due to the lack of explicit assumptions about the properties of $\varphi$ beyond its homomorphism property. For a rigorous proof, additional information about $\varphi$ (e.g., linearity, effects on rank, or specific actions on basis elements) would be necessary.

\end{proof}

\begin{proof}[Solution (o1-2024-12-17)]
    We shall prove the two statements in turn.  Throughout the proof, let $\varphi\colon \mathbb{R}^{n\times n} \to \mathbb{R}^{n\times n}$ be a map (not assumed a priori to be unital) satisfying  
\[
\varphi(AB) = \varphi(A)\,\varphi(B),
\]
for all $A,B\in \mathbb{R}^{n\times n}$.  Recall that $\Gamma_r$ denotes the set of all $n\times n$ real matrices of rank $r$, and $E_{ij}$ is the matrix having a $1$ in the $(i,j)$-entry and $0$ elsewhere.


A key tool in the argument is the fact that $M_n(\mathbb{R})$ (the ring of all $n\times n$ real matrices) is a simple ring (over the field $\mathbb{R}$).  Being ``simple'' means that $M_n(\mathbb{R})$ has no nontrivial two-sided ideals (the only two-sided ideals in $M_n(\mathbb{R})$ are $\{0\}$ and $M_n(\mathbb{R})$ itself).

Given a ring-homomorphism $\varphi\colon M_n(\mathbb{R}) \to M_n(\mathbb{R})$ (where ``ring-homomorphism'' here means preserving multiplication, but not necessarily preserving the identity unless one checks $\varphi(I)$), one knows from standard ring theory:

\begin{itemize}
\item Either $\varphi$ is the zero map (i.e. $\varphi(A)=0$ for all $A$),
\item Or $\varphi$ is injective, because the kernel of $\varphi$ is a two-sided ideal in a simple ring $M_n(\mathbb{R})$, hence must be either $\{0\}$ or all of $M_n(\mathbb{R})$.  If $\varphi$ is not identically zero, its kernel cannot be all of $M_n(\mathbb{R})$, so its kernel is $\{0\}$, and $\varphi$ is injective.
\end{itemize}

We shall use this fact repeatedly.  Whenever we refer to ``simplicity of $M_n(\mathbb{R})$'', we are using precisely that the only two-sided ideals are $\{0\}$ or $M_n(\mathbb{R})$.


Statement (1) says:

\[
\text{``If $A$ and $B$ are both rank-$r$ matrices, then $\varphi(A)$ and $\varphi(B)$ have the same rank.''}
\]

Equivalently, $\varphi$ is ``constant on ranks'': for each $r=0,1,\dots,n$ there is some value (say $\alpha_r$) so that $\operatorname{rank}(\varphi(A))=\alpha_r$ for every $A$ of rank $r$.

\textbf{STEP 1A.}  Either $\varphi$ is the zero map, or else it is injective.

By the simplicity argument above, the kernel of $\varphi$ is a two-sided ideal of $M_n(\mathbb{R})$.  If $\varphi$ is not identically zero, that kernel must be $\{0\}$, so $\varphi$ becomes injective.

\begin{itemize}
\item \textit{Case (i):} $\varphi$ is the zero map.  
Then trivially $\operatorname{rank}(\varphi(A))=0$ for all $A$, so if $A,B$ happen to have rank $r$, indeed $\operatorname{rank}(\varphi(A))=\operatorname{rank}(\varphi(B))=0$.  In this case, statement (1) is true (trivially).

\item \textit{Case (ii):} $\varphi$ is nonzero and hence injective.  
We shall show that, in this scenario, $\varphi$ sends any two matrices of the same rank to matrices whose ranks coincide.
\end{itemize}

\textbf{STEP 1B.} $\varphi(I)$ must be an idempotent, and either $\varphi(I)=0$ or $\varphi(I)$ acts as an identity in the image.

Because $I\cdot I = I$, applying $\varphi$ gives 
\[
\varphi(I)\varphi(I)=\varphi(I).
\]
So $\varphi(I)$ is an idempotent.  An idempotent $E$ in $M_n(\mathbb{R})$ either has some nonzero rank (possibly $<n$) or can be zero.

\begin{itemize}
\item If $\varphi(I)=0$, then for any matrix $A$,
\[
\varphi(A)=\varphi(AI)=\varphi(A)\varphi(I)=\varphi(A)\cdot 0=0.
\]
That forces $\varphi$ to be the zero map on all of $M_n(\mathbb{R})$, contradicting that we are now in the nonzero case.  Hence in the nonzero/injective scenario, $\varphi(I)$ cannot be zero.
\end{itemize}

Therefore $\varphi(I)\neq 0$ is an idempotent in the image ring $\varphi(M_n(\mathbb{R}))$.  In this situation (a standard lemma in ring theory) one shows that $\varphi(I)$ must actually serve as the identity element of the subring $\varphi(M_n(\mathbb{R}))$.  Concretely, for every $A$,
\[
\varphi(A)=\varphi(AI)=\varphi(A)\varphi(I),
\]
so $\varphi(I)$ acts like a multiplicative identity on the right of $\varphi(A)$.  A similar argument on the left shows $\varphi(I)$ is a two-sided identity.

Thus in the nonzero case, $\varphi$ is a ``unital homomorphism'' onto its image.

\textbf{STEP 1C.} If $A$ is invertible, then $\varphi(A)$ must be invertible (in the subring $\varphi(M_n(\mathbb{R}))$).

Suppose $A\in M_n(\mathbb{R})$ is invertible with inverse $A^{-1}$.  Then 
\[
A A^{-1} = I,
\]
so  
\[
\varphi(A)\varphi(A^{-1}) = \varphi(I),
\]
the identity in the image.  Hence $\varphi(A)$ has a two-sided inverse $\varphi(A^{-1})$, so $\varphi(A)$ is invertible.

Consequently, whenever $A$ is an invertible $n\times n$ matrix, the image $\varphi(A)$ also is invertible (and thus has rank $n$).

\textbf{STEP 1D.} Conclude that matrices of a fixed rank $r$ are sent to matrices all of the same rank.

We now show that any two rank-$r$ matrices $A,B$ get sent to matrices of the same rank under $\varphi$.

\textbf{FACT:} Any two rank-$r$ matrices $A,B\in M_n(\mathbb{R})$ can be related via invertible left–right factors.  
Indeed, if $\operatorname{rank}(A)=r$, one can write
\[
A = P D Q,
\]
where $P,Q\in GL_n(\mathbb{R})$ and $D=\mathrm{diag}(I_r,0,\dots,0)$.  
Similarly, $B=P' D Q'$ for some invertible $P',Q'$.  
Thus
\[
B = (P' P^{-1})\, A\, (Q^{-1} Q').
\]

Apply $\varphi$ to $B$:
\[
\varphi(B)=\varphi(R A S)=\varphi(R)\,\varphi(A)\,\varphi(S).
\]
Since (by Step 1C) $\varphi(R)$ and $\varphi(S)$ are invertible (in the image), multiplying $\varphi(A)$ on the left and right by invertibles does not change the rank.  Therefore
\[
\operatorname{rank}(\varphi(B))=\operatorname{rank}(\varphi(A)).
\]

Hence any two rank-$r$ matrices must have images under $\varphi$ of the same rank:
\[
\forall\, A,B\in \Gamma_r,\qquad \operatorname{rank}(\varphi(A))=\operatorname{rank}(\varphi(B)).
\]

This completes the proof of (1).


Statement (2) says:

Assume additionally $\varphi(0)=0$ and that there exists a rank-1 matrix $W$ with $\varphi(W)\neq 0$. Then there exists an invertible matrix $R$ such that 
$\varphi(E_{ij})=R E_{ij} R^{-1}\text{ for all } i,j.$

That is, under those extra conditions, $\varphi$ is actually an inner automorphism of the matrix algebra.

\textbf{STEP 2A.} $\varphi$ is nonzero, hence injective and unital as in (1).

Since $W$ has rank $1$ and $\varphi(W)\neq 0$, $\varphi$ is not the zero map.  By simplicity, $\varphi$ is injective, and by Step 1B, $\varphi$ is unital onto its image.

\textbf{STEP 2B.} Minimal idempotents of rank $1$ go to pairwise orthogonal idempotents of rank $1$.

The standard $E_{ii}$ each have rank $1$.  By part (1), all $\varphi(E_{ii})$ have the same rank, and the existence of $W$ forces this rank to be at least $1$.  Since $\varphi(E_{ii})$ is an idempotent, its rank equals its trace, so all must be rank $1$.

If $i\neq j$, then $E_{ii} E_{jj}=0$, hence $\varphi(E_{ii})\varphi(E_{jj})=\varphi(0)=0$, so the rank-$1$ idempotents are orthogonal.

\textbf{STEP 2C.} Decomposing the identity: $\sum_i E_{ii}=I \implies \sum_i \varphi(E_{ii})=\varphi(I)$.

Because $\varphi(I)$ is the identity in the image, the sum of the $n$ orthogonal rank-$1$ idempotents is the $n$-dimensional identity.  Thus $\mathbb{R}^n$ decomposes into $n$ one-dimensional invariant subspaces.

Pick a nonzero $v_i$ in the image of $\varphi(E_{ii})$ for each $i$.  
Then $\{v_1,\dots,v_n\}$ is a basis.

\textbf{STEP 2D.} Define $R$ so that conjugation by $R$ agrees with $\varphi$ on each $E_{ij}$.

Let $e_1,\dots,e_n$ be the standard basis.  
Define $R$ by
\[
R(e_i)=v_i.
\]
This is invertible.

We show $\varphi(E_{ij})=R E_{ij} R^{-1}$.

$E_{ij}$ satisfies $E_{ij}(e_k)=0$ unless $k=j$, and $E_{ij}(e_j)=e_i$.

Conjugation gives
\[
(R E_{ij} R^{-1})(v_k)=R E_{ij}(e_k)=
\begin{cases}
0, & k\neq j,\\
v_i, & k=j.
\end{cases}
\]

Similarly, $\varphi(E_{ij})$ kills all $v_k$ with $k\neq j$, and $\varphi(E_{ij})(v_j)$ lies in the span of $v_i$, with scalar $1$ after normalization.

Thus $\varphi(E_{ij})$ and $R E_{ij} R^{-1}$ act identically on the basis $\{v_1,\dots,v_n\}$, hence are equal.

This completes the proof of (2).


\textbf{1)} We used that $M_n(\mathbb{R})$ is a simple ring.  
This gives the dichotomy ``either $\varphi$ is zero or injective''.

\textbf{2)} We used that if $\varphi$ is injective and sends $1$ to an idempotent $\varphi(1)$, then either $\varphi(1)=0$ (forcing $\varphi$ to be zero) or else $\varphi(1)$ serves as the identity in the image.

\textbf{3)} The step (1C) follows from $A A^{-1}=I \implies \varphi(A)\varphi(A^{-1})=\varphi(I)$.


Putting all pieces together:

\begin{itemize}
\item \textbf{Part (1):} For each rank $r$, all $A\in \Gamma_r$ get mapped to matrices of a single rank.
\item \textbf{Part (2):} Under $\varphi(0)=0$ and the existence of $W$ of rank $1$ with $\varphi(W)\neq 0$, $\varphi$ is an inner automorphism:  
\[
\varphi(E_{ij}) = R E_{ij} R^{-1}.
\]
\end{itemize}

This completes the required proofs.

\end{proof}

\textbf{Comment}: The GPT-4 solution correctly notes that multiplicativity alone does not immediately imply rank preservation, but it ultimately does not provide proofs of (1) or (2). The o1 solution contains several substantial gaps: it treats \(\ker\varphi\) as a two-sided ideal (which requires additivity), uses ``simplicity" to deduce injectivity, and later applies additivity when manipulating \(\sum_i E_{ii}=I\). Moreover, invertibility in the image algebra with identity \(\varphi(I)\) does not automatically imply invertibility in \(M_n(\mathbb R)\), so the rank arguments need further justification.

\begin{thm}
   Let $P(x)$ be a real polynomial of degree $n$ , which satisfies $\int_a^bP(x)x^kdx=0,k=0,\cdots,n-1$. (1)Prove that $P(x)$ has $n$ real simple roots at $(a,b)$. (2)Prove that for an $n$-point integration formula on the interval $[a, b]$, if the roots of the equation $P(x) = 0$ are used as nodes, the algebraic precision can reach up to $2n-1$ order. 
\end{thm}

\begin{proof}[Solution (GPT-4)]
    To address this problem, we'll tackle it in two parts as requested.

 Part 1: Proving $P(x)$ has $n$ real simple roots in $(a, b)$

Given that $P(x)$ is a real polynomial of degree $n$ satisfying the condition $\int_a^b P(x)x^k dx = 0$ for $k = 0, \cdots, n-1$, we aim to show that $P(x)$ has $n$ real simple roots in the interval $(a, b)$.

\textbf{Proof:}

1. \textbf{Orthogonality Principle}: The condition $\int_a^b P(x)x^k dx = 0$ for $k = 0, \cdots, n-1$ implies that $P(x)$ is orthogonal to all polynomials of degree less than $n$ on the interval $[a, b]$. This is a key observation derived from the properties of inner products in the vector space of polynomials.

2. \textbf{Existence of Roots}: Assume for contradiction that $P(x)$ does not have $n$ real roots in $(a, b)$. Since $P(x)$ is a polynomial of degree $n$, it must have exactly $n$ roots in the complex plane (counting multiplicities) by the Fundamental Theorem of Algebra. If not all of these roots are real and simple within $(a, b)$, then either some roots are complex, or they are real but outside $(a, b)$, or some roots are multiple.

3. \textbf{Role of Orthogonality}: If $P(x)$ does not vanish at $n$ distinct points in $(a, b)$, then there exists at least one polynomial $Q(x)$ of degree less than $n$ such that $Q(x)$ does not change sign on $(a, b)$ where $P(x)$ does not vanish. This implies $\int_a^b P(x)Q(x) dx \neq 0$, contradicting the given orthogonality condition.

4. \textbf{Conclusion}: Therefore, $P(x)$ must have $n$ real, simple roots in the interval $(a, b)$. This conclusion follows from the contradiction in step 3 and the fact that a polynomial of degree $n$ cannot be orthogonal to a non-zero polynomial of lower degree unless it has specific sign changes that allow the integral to vanish.

 Part 2: Proving the Algebraic Precision of an $n$-point Integration Formula

Given an $n$-point integration formula on the interval $[a, b]$ using the roots of $P(x) = 0$ as nodes, we need to prove that its algebraic precision can reach up to $2n-1$ order.

\textbf{Proof:}

1. \textbf{Gauss Quadrature Principle}: The $n$-point Gauss quadrature formula is exact for polynomials of degree up to $2n-1$. This is because it uses the roots of orthogonal polynomials (with respect to a weight function, which is $1$ in our case) as nodes. Our polynomial $P(x)$, satisfying the given conditions, acts similarly to these orthogonal polynomials.

2. \textbf{Construction of the Quadrature}: Since the roots of $P(x)$ are used as nodes, and we've established that $P(x)$ has $n$ real, simple roots within $(a, b)$, these nodes are the optimal choice for constructing a quadrature formula with the highest possible degree of precision.

3. \textbf{Precision of the Quadrature}: The precision of an integration formula is determined by the highest degree of polynomial it can integrate exactly. By the construction of Gauss quadrature and the properties of orthogonal polynomials, it's known that using the roots of such polynomials as nodes yields a formula that is exact for polynomials of degree up to $2n-1$.

4. \textbf{Conclusion}: Therefore, the $n$-point integration formula using the roots of $P(x)$ as nodes has an algebraic precision of up to $2n-1$ order. This is a direct consequence of the properties of orthogonal polynomials and the Gauss quadrature principle.

In summary, the given conditions on $P(x)$ ensure it behaves similarly to orthogonal polynomials used in Gauss quadrature, allowing us to conclude both the existence of $n$ real simple roots within $(a, b)$ and the high degree of precision for the corresponding $n$-point integration formula.

\end{proof}

\begin{proof}[Solution (o1-2024-12-17)]

We are given a real polynomial $P(x)$ of degree $n$ on the interval $[a,b]$ satisfying the orthogonality conditions  
\[
\int_a^b P(x)\, x^k\, dx = 0 \quad \text{for } k = 0,1,2,\dots,n-1.
\]

There are two main claims to prove:

1) $P(x)$ has $n$ real, simple (i.e., distinct) roots in $(a,b)$.  
2) If one uses these $n$ roots as the nodes of an $n$-point quadrature rule on $[a,b]$, then that rule is exact (i.e., has ``algebraic precision'') for all polynomials of degree $\le 2n-1$.

\textbf{PART (1).  $P(x)$ HAS $n$ REAL, SIMPLE ROOTS IN $(a,b)$}

\textbf{Step 1A.} $P(x)$ must change sign at least $n$ times in $(a,b)$.

\begin{itemize}
\item First, observe that $P(x)$ cannot remain strictly nonnegative nor strictly nonpositive on the entire interval $(a,b)$.
    \begin{itemize}
    \item If $P(x)\ge 0$ for all $x\in (a,b)$, then the integral $\int_a^b P(x)\, dx$ would be strictly positive (unless $P$ is the zero polynomial, which it is not since it has degree $n$).  This contradicts the orthogonality condition for $k=0$:
    \[
    \int_a^b P(x)\cdot 1\, dx = 0.
    \]
    \item A similar argument shows it cannot be strictly nonpositive either.
    \end{itemize}
\end{itemize}

Hence, $P(x)$ must take on both positive and negative values somewhere in $(a,b)$.  By the Intermediate Value Theorem, each sign change requires at least one real root.  Next, we show that there must be at least $n$ sign changes in $(a,b)$, forcing at least $n$ roots.

\textbf{Step 1B.} Constructing a contradiction if there are fewer than $n$ roots.

Assume for contradiction that $P(x)$ has at most $n-1$ real roots in $(a,b)$.  Then it can have at most $n-1$ sign changes.

We now construct a polynomial $Q(x)$ of degree $<n$ whose integral against $P(x)$ does not vanish.  Roughly, we let $Q(x)$ ``follow'' the sign of $P(x)$ so that $P(x)Q(x)$ is nonnegative or nonpositive and strictly nonzero on large subintervals.

Sketch of the construction of $Q(x)$:
\begin{itemize}
\item Partition $(a,b)$ into subintervals where $P(x)$ is positive or negative.
\item Pick one point in each subinterval. Call them $x_1,\dots,x_m$, with $m\le n-1$.
\item Define $Q(x)$ of degree $\le m-1$ (hence $\le n-2 < n$) so that
\[
Q(x_i)=
\begin{cases}
+1 & \text{if } P(x_i)>0,\\
-1 & \text{if } P(x_i)<0.
\end{cases}
\]
Interpolation guarantees this is possible.
\end{itemize}

Because $Q(x_i)$ has the same sign as $P(x_i)$, on each subinterval the product $P(x)Q(x)$ keeps one sign and does not integrate to zero. Thus
\[
\int_a^b P(x)Q(x)\, dx \ne 0,
\]
contradicting orthogonality to all polynomials of degree $<n$.

Hence $P(x)$ must have at least $n$ real roots in $(a,b)$.

\textbf{Step 1C.} $P(x)$ must be of exact degree $n$, so it has exactly $n$ real roots.

A degree-$n$ polynomial has at most $n$ real roots.  Combined with at least $n$ roots, we obtain exactly $n$ roots.

\textbf{Step 1D.} The roots are distinct (no multiple roots).

A standard result for orthogonal polynomials states:
\begin{itemize}
\item For an orthogonal polynomial $P(x)$ (with respect to a positive weight), the roots of $P$ and $P'$ interlace.
\item Thus $P$ and $P'$ cannot share a root unless $P$ is identically zero, so $P$ has no repeated roots.
\end{itemize}

Alternatively, one may argue directly that a repeated root leads to fewer sign changes or produces a contradiction by constructing an appropriate $Q(x)$.

Thus $P(x)$ has $n$ distinct real roots in $(a,b)$.

\textbf{PART (2).  A $2n-1$ ORDER QUADRATURE RULE AT THE ROOTS}

We now prove that if one takes the $n$ distinct real roots of $P(x)$ as the nodes of an $n$-point formula on $[a,b]$, then this formula is exact for all polynomials of degree up to $2n-1$.

Recall the statement:

An $n$-point quadrature rule has algebraic precision $m$ if 
$\int_a^b R(x)\, dx = \sum_{i=1}^n w_i R(x_i)$
\quad for all $\deg(R)\le m$.

We will show $m=2n-1$ is attainable with nodes $x_i$ equal to the roots of $P$.

\textbf{Step 2A.} Choosing weights $w_i$ to integrate all polynomials of degree $\le n-1$.

Let $x_1,\dots,x_n$ be the $n$ distinct real roots of $P(x)$.  We seek $w_1,\dots,w_n$ such that
\[
\int_a^b R(x)\, dx = \sum_{i=1}^n w_i R(x_i)
\]
for every polynomial $R$ of degree $\le n-1$.

It suffices to impose this on the basis $1,x,x^2,\dots,x^{n-1}$.  This gives $n$ linear equations for $n$ unknowns $w_i$.  The distinctness of $x_i$ ensures the system is nonsingular (Vandermonde-like), so the solution exists and is unique.

\textbf{Step 2B.} Exactness up to degree $2n-1$ follows from orthogonality.

Let $Q(x)$ be any polynomial with $\deg(Q)\le 2n-1$.  Perform polynomial division:
\[
Q(x)=A(x)+B(x)P(x),
\]
where $\deg(A)<n$ and $\deg(B)\le n-1$.

Compute the integral:
\[
\int_a^b Q(x)\, dx = \int_a^b A(x)\, dx + \int_a^b B(x)P(x)\, dx.
\]
Orthogonality gives
\[
\int_a^b B(x)P(x)\, dx = 0,
\]
because $\deg(B)\le n-1$.  Thus
\[
\int_a^b Q(x)\, dx = \int_a^b A(x)\, dx.
\]

Now compute the quadrature sum:
\[
\sum_{i=1}^n w_i Q(x_i)
= \sum_{i=1}^n w_i \bigl[A(x_i)+B(x_i)P(x_i)\bigr].
\]
But $P(x_i)=0$ for all $i$, so this equals
\[
\sum_{i=1}^n w_i A(x_i).
\]
By construction of $w_i$,
\[
\sum_{i=1}^n w_i A(x_i)=\int_a^b A(x)\, dx.
\]

Therefore,
\[
\sum_{i=1}^n w_i Q(x_i)
= \int_a^b Q(x)\, dx.
\]

Thus the quadrature rule is exact for all polynomials of degree at most $2n-1$.

\textbf{CONCLUSION}

1) The orthogonality conditions $\int_a^b P(x)x^k\,dx=0$ for $k=0,\dots,n-1$ force $P(x)$ to change sign at least $n$ times in $(a,b)$.  Since $P$ is degree $n$, it has exactly $n$ distinct real roots in $(a,b)$.

2) Using these $n$ roots as quadrature nodes and choosing weights $w_i$ appropriately yields a rule exact for all polynomials of degree up to $2n-1$.  Any polynomial $Q$ of degree $\le 2n-1$ decomposes as $Q=A+BP$, orthogonality kills the integral of $BP$, and the root property kills its quadrature contribution.

This completes the proofs of (1) and (2).

\textbf{Remark (How theorems are used).}
\begin{itemize}
\item The Intermediate Value Theorem guarantees real roots when a continuous function changes sign.
\item A theorem from orthogonal polynomial theory ensures $P$ and $P'$ cannot share roots, so no repeated roots.
\item The argument for degree $2n-1$ exactness is the classical Gaussian quadrature proof using polynomial division and orthogonality.
\end{itemize}

\end{proof}

\textbf{Comment}: For GPT-4’s solution, the idea for the first question is correct, but the derivation of the contradiction is not explained in sufficient detail. The second question is not addressed at all; the response simply restates the problem without providing a proof. For o1-2024-12-17’s solution, the argument in Part (1) correctly uses orthogonality to show that the polynomial must change sign multiple times, leading to the conclusion that it has $n$ real, distinct roots. The proof of Part (2) is also correct.

\end{document}